\documentclass[10pt]{amsart}
\usepackage[utf8]{inputenc}
\usepackage{amsmath, latexsym, amsfonts, amssymb, amsthm,enumitem}
\usepackage{graphicx, color,dsfont,epsfig,caption,wrapfig,subfig}
\usepackage{stmaryrd}
\usepackage{tikz}
\usepackage{datetime}

\usepackage{hyperref}
 \hypersetup{
     linktoc=page,
     linkcolor=red,          
     citecolor=blue,        
     filecolor=blue,      
     urlcolor=cyan,
     colorlinks=true           
 }

\renewcommand{\epsilon}{\varepsilon}
\newcommand{\cA}{{\ensuremath{\mathcal A}} }

\newcommand{\cO}{{\ensuremath{\mathcal O}} }
\newcommand{\cF}{{\ensuremath{\mathcal F}} }

\newcommand{\cC}{{\ensuremath{\mathcal C}} }

\newcommand{\cM}{{\ensuremath{\mathcal M}} }
\newcommand{\cZ}{{\ensuremath{\mathcal Z}} }

\newcommand{\cS}{{\ensuremath{\mathcal S}} }

\newcommand{\bi}{{\ensuremath{\mathbf i}} }

\newcommand{\bP}{{\ensuremath{\mathbf P}} }
\newcommand{\bE}{{\ensuremath{\mathbf E}} }

\newcommand{\bbC}{{\ensuremath{\mathbb C}} }

\newcommand{\bbE}{{\ensuremath{\mathbb E}} }

\newcommand{\bbP}{{\ensuremath{\mathbb P}} }

\newcommand{\bbR}{{\ensuremath{\mathbb R}} }

\newcommand{\gep}{\varepsilon}       

\newcommand{\gO}{\Omega}
\numberwithin{equation}{section}

\newtheorem{theorem}{Theorem}[section]

\newtheorem{lemma}[theorem]{Lemma}
\newtheorem{proposition}[theorem]{Proposition}
\newtheorem{corollary}[theorem]{Corollary}
\newtheorem{rem}[theorem]{Remark}
\newtheorem{remark}[theorem]{Remark}

\newtheorem{assumption}[theorem]{Assumption}

\theoremstyle{definition}

\renewcommand{\tilde}{\widetilde}          
\DeclareMathSymbol{\leqslant}{\mathalpha}{AMSa}{"36} 
\DeclareMathSymbol{\geqslant}{\mathalpha}{AMSa}{"3E} 
\DeclareMathSymbol{\eset}{\mathalpha}{AMSb}{"3F}     
\renewcommand{\leq}{\;\leqslant\;}                   
\renewcommand{\geq}{\;\geqslant\;}                   
\newcommand{\dd}{\text{\rm d}}             

\newcommand{\maxtwo}[2]{\max_{\substack{#1 \\ #2}}} 
\newcommand{\mintwo}[2]{\min_{\substack{#1 \\ #2}}} 

\newcommand{\R}{\mathbb{R}}

\newcommand{\N}{\mathbb{N}}

\newcommand{\E}{\mathds{E}}
\renewcommand{\P}{\mathds{P}}
\newcommand{\ind}{\mathds{1}}

\newcommand{\bx}{\mathbf{x}}

\usepackage{xparse}

\DeclareDocumentCommand \Pmp { m m o} {
\IfNoValueTF{#3}
{P_{#1}^{#2}}
{P_{#1}^{#2}\left(#3\right)}
}

\DeclareDocumentCommand \Emp { m m o} {
\IfNoValueTF{#3}
{E_{#1}^{#2}}
{E_{#1}^{#2}\left[#3\right]}
}

\DeclareDocumentCommand \Pbr { m m m m o } {
\IfNoValueTF{#5}
{P_{#1}^{#2\stackrel{#4}{\rightarrow} #3}}
{P_{#1}^{#2\stackrel{#4}{\rightarrow} #3}\left(#5\right)}
}

\DeclareDocumentCommand \Ebr { m m m m o } {
\IfNoValueTF{#5}
{E_{#1}^{#2\stackrel{#4}{\rightarrow} #3}}
{E_{#1}^{#2\stackrel{#4}{\rightarrow} #3}\left[#5\right]}
}

\newcommand{\Cov}{\mathrm{Cov}}
\newcommand{\Var}{\mathrm{Var}}

\def\bnum{\begin{enumerate}}
\def\enum{\end{enumerate}}
\def\<#1{\langle #1 \rangle}

\def\cF{\mathcal{F}}

\newcommand{\gb}{\beta}
\newcommand{\gd}{\delta}

\newcommand{\go}{\omega}
\newcommand{\gl}{\lambda}

\newcommand{\esssup}{\mathrm{ess\,sup}}
\renewcommand{\hat}{\widehat}

\newcommand{\lint}{\llbracket}
\newcommand{\rint}{\rrbracket}

\begin{document}

\title[Ultraviolet renormalization of the boundary Sine-Gordon model]{A probabilistic approach of ultraviolet renormalization in the boundary Sine-Gordon model}

\author[Hubert Lacoin]{Hubert Lacoin }
\address{IMPA, Estrada Dona Castorina 110, Rio de Janeiro, RJ-22460-320, Brasil. Supported by FAPERj (grant JCNE) and CPNq (grant Universal and productivity grant)}
\email{ }

\author[R\'emi Rhodes]{R\'emi Rhodes}
\address{Aix-Marseille university, Institut de math\'ematiques I2M, Technop\^ole Ch\^ateau-Gombert, 39 rue F. Joliot Curie, 13453 Marseille Cedex 13, France. Partially supported by grant    ANR-15-CE40-0013 Liouville.}

\author[Vincent Vargas]{ Vincent Vargas}
\address{ENS Ulm, DMA, 45 rue d'Ulm,  75005 Paris, France. Partially supported by grant    ANR-15-CE40-0013 Liouville.}




 \begin{abstract}
The  Sine-Gordon model  is obtained by tilting the law of  a log-correlated Gaussian field $X$ defined on a subset of $\bbR^d$ by the exponential of its cosine, namely $\exp(\alpha \smallint  \cos (\beta X))$.  It has gathered significant attention due to its importance in Quantum Field Theory and to its connection with  the study of   log-gases in Statistical Mechanics. In spite of its relatively simple definition, the model has a very rich phenomenology. While the  integral $\smallint  \cos (\beta X)$ can be defined properly when $\beta^2<d$   using the standard Wick normalization of $\cos (\beta X)$, a more involved renormalization procedure is needed when $\beta^2\in [d,2d)$. In particular it exhibits a countable sequence of phase transitions accumulating to the left of $\beta=\sqrt{2d}$,  each transition  corresponding to the addition of an extra term in the renormalization scheme. The final threshold $\beta=\sqrt{2d}$ corresponds to the Kosterlitz-Thouless (KT) phase transition of the $\log$-gas. In this paper, we present a novel probabilistic approach to renormalization of the two-dimensional boundary (or 1-dimensional) Sine-Gordon model up to the KT threshold $\beta=\sqrt{2d}$. The purpose of this approach is to propose a simple and flexible method  to treat this  problem which, unlike the existing renormalization group techniques, does not  rely on translation invariance for the covariance kernel of $X$ or the reference measure along which  $\cos (\beta X)$ is integrated. To this purpose we establish by induction a general formula for the cumulants of a random variable  defined on a filtered probability space expressed in terms of brackets of a family of martingales; to the best of our knowledge, the recursion formula is new and might have other applications.  We apply this  formula to study the cumulants of (approximations of) $\smallint  \cos (\beta X)$. To control all terms produced by the induction procedure, we prove a refinement of classical electrostatic inequalities, which allows us to bound the energy of configurations in terms of the Wasserstein distance between $+$ and $-$ charges.

\end{abstract}

\maketitle

\section{Introduction}

\subsection{Log-Gases, $2D$-Yukawa Gas and $2D$-Coulomb Gas}
The $d$-dimensional \textit{log-gas} is a model of statistical mechanics   describing a gas of interacting charged particles, the interaction being given by a potential $K(x,y)$, where $K$ is continuous outside of the diagonal   and $K(x,y)\stackrel{|x-y|\to 0}{=} - \ln |x-y|+O(1)$.

\medskip

Given two positive parameters $\alpha, \beta>0$, $\alpha$ standing for  particle activity and $\beta^2$ for inverse temperature (the reason for such an exotic parametrization appears below),
its formal partition function    over a bounded open set $\mathcal{O}\subset\R^d$ is given by
\begin{equation}\label{2DCGapprox}
\mathcal{Z}^K_{\alpha,\beta}(\mathcal{O})=\sum_{n= 0}^{\infty} \frac{\alpha^n}{n!}\sum_{(\gl_i)^n_{i=1}\in \{-1,1\}^n}
\int_{\mathcal{O}^{n}}\exp\Big(- \beta^2 \sum_{1\le i<j\le n}\lambda_i\lambda_jK(x_i,x_j) \Big) \prod_{i=1}^{n}\dd x_i.
\end{equation}
The above integral diverges when $\gb^2\ge d$. In this case, to have an interpretation of the above expression, we need to consider a reasonable sequence of bounded approximations $K_{\gep}$ of the potential $K$  and consider the limit
\begin{equation}
 \mathcal{Z}^K_{\alpha,\beta}(\mathcal{O})=\lim_{\gep \to 0}\mathcal{Z}^K_{\alpha,\beta,\gep}(\mathcal{O})
\end{equation}
where 
\begin{equation}\label{2DCGapproxics}
\mathcal{Z}^K_{\alpha,\beta,\gep}(\mathcal{O}):=\sum_{n= 0}^{\infty} \frac{\alpha^n}{n!}\sum_{(\gl_i)^n_{i=1}\in \{-1,1\}^n}
\int_{\mathcal{O}^{n}}\exp\Big( -\beta^2 \sum_{1\le i<j\le n}\lambda_i\lambda_j K_{\gep} (x_i,x_j)  \Big) \prod_{i=1}^{n}\dd x_i.
\end{equation}
While, of course, strictly speaking, $\mathcal{Z}^K_{\alpha,\beta,\gep}$ also diverges when $\gep$ tends to $0$ if $\gb^2\ge  d$,
 a non-trivial interpretation of the limit is obtained by dividing by an appropriate (and also diverging) counter-term. This renormalization scheme is the first step needed to define the Gibbs measure associated with the partition function $\mathcal{Z}^K_{\alpha,\beta}(\mathcal{O})$ and study its properties.

\medskip

A prototypical example of log-gas is the $2D$-Yukawa gas ($d=2$) with mass $m>0$, corresponding to the potential 
$$K_m(x,y)=\int^{\infty}_0 \frac{1}{2t}e^{-\frac{|x-y|^2}{2t}-mt}\dd t.$$
Another example of central interest is
the two-dimensional (free boundary) Coulomb gas, corresponding to taking the limit
of the Yukawa potential when $m$ goes to zero.
This formally corresponds to taking  $K(x,y)= - \ln |x-y| +\infty$, and has the effect of giving an infinite energy to any configuration with non-zero global charge. This corresponds to 
\begin{equation}\label{2DCG}
\mathcal{Z}^{\mathrm{Coulomb}}_{\alpha,\beta}(\mathcal{O})=\sum_{n= 0}^{\infty}\frac{\alpha^{2n}}{(n!)^2}\int_{\mathcal{O}^{2n}}\exp\Big( \beta^2 \sum_{i<j}\lambda^{(n)}_i\lambda^{(n)}_j\ln |x_i-x_j| \Big)\prod_{i=1}^{2n}\dd x_i,
\end{equation}
where $\gl^{(n)}_i= -1$ if $1\le i\le n$ and $\gl^{(n)}_i=1$ if $n+1\le i\le 2n$.
Again $\ln$ has to be replaced e.g. by $\ln_\gep r:= \frac{1}{2} \ln(r^2+ \gep^2)$ in order to take the limit when $\gep\to 0$ and thus to give an appropriate renormalization scheme.
One could  as well consider for $K$ other types of Green functions with various boundary conditions, leading to  Coulomb gases with various boundary conditions.

\subsection{The Sine-Gordon representation}

A well known tool to identify the proper renormalization for log-gases (including Coulomb and Yukawa) is their Sine-Gordon representation, which we briefly sketch now. 
To define it we need to assume that $K$ is positive semi-definite in the sense
that for any bounded continuous function on $\mathcal{O}$ we have 
\begin{equation}\label{prosit}
\int_{\mathcal{O}^2} K(x,y) g(x)g(y) \dd x \dd y\ge 0.
\end{equation}
It is standard to check that the assumption is satisfied for the Yukawa potential. Concerning the Sine-Gordon representation of the (free boundary) Coulomb gas let us mention only that $K_R(x,y)=-\ln(|x-y|/R)$  is positive semidefinite on $\mathcal{O}^2$ if  $\mathcal{O}\subset B(0,R)$.
Consider a centered Gaussian Field $X$ defined on $\cO$,  with law $\P$ and expectation $\E$, with  covariance function
$$\E[X(x)X(y)]=K(x,y),\quad \text{ for }x,y\in \mathcal{O}.$$
The condition \eqref{prosit} guarantees that such a field $X$ can be defined as a random distribution on $\cO$ \ : \  the pointwise value $X(x)$ does not make sense but $\int_{\cO} X(x) g(x) \dd x$ does for all sufficiently regular functions $g$.
The Sine-Gordon model is formally defined as the measure 
\begin{equation}\label{defSGintro}
\exp\Big(2\alpha   \int_{\mathcal{O}}\cos(\beta X(x)) \,\dd x  \Big) \P(\dd X).
\end{equation}
Since $X$ is not a function but a random distribution, the mathematical interpretation of \eqref{defSGintro} requires further explanations. A proper definition requires renormalization  :  more specifically,  the field $X$ can be smoothened by convolution with a reasonable (e.g. compactly supported and $C^{\infty}$) isotropic kernel (in the paper we use another type of approximation of $X$ which is not a convolution but this is not relevant for this part of the discussion). We call $X_\epsilon$ the regularized field and $K_{\gep}(x,y)$ the corresponding covariance (we just write $K_{\gep}(x)$ for the variance). The Sine-Gordon model is defined as the limit as $\epsilon$ goes to $0$   of the following well-defined sequence of measures 
\begin{equation}\label{defSGintroepsilon}
\lim_{\epsilon \to 0}  \, \exp\Big(2\alpha   \int_{\mathcal{O}}\cos(\beta X_\epsilon (x))  e^{\frac{\beta^2}{2}K_{\gep}(x)} \,\dd x \Big)  \P(\dd X).
\end{equation}
This measure is directly related to the partition  function of the $d$-dimensional  log-gas. Indeed, complex exponential moments of the Gaussian field are related to the partition function of the $d$-dimensional  log gas with fixed number of particles through the following relation  :   for any charge distribution $(\gl_i)_{i=1}^n$ with $k$ positive charges and $n-k$ negative ones, we have the following correspondence 
\begin{equation}\label{momCG}
\E\Big[\Big(\int_{\mathcal{O}}\!\!e^{\mathbf{i}\beta X_\epsilon(x)+\frac{\beta^2}{2}K_{\gep}(x)}\,\dd x\Big)^{k}\Big(\int_{\mathcal{O}}\!\!e^{-\mathbf{i}\beta X_\epsilon(x)+\frac{\beta^2}{2}K_{\gep}(x)}\,\dd x\Big)^{n-k}\Big]\\= \int_{\mathcal{O}^{n}} \!\!\!e^{- \beta^2 \sum_{1\le i<j\le n}\lambda_i\lambda_j K_\epsilon(x_i,x_j) }\prod_{i=1}^{n}\dd x_i.
\end{equation}
As a consequence of \eqref{momCG} we have
\begin{multline}\label{asdf}
\bbE\left[\exp\big(2\alpha  \int_{\mathcal{O}}\cos(\beta X_\epsilon (x)) e^{\frac{\beta^2}{2}K_{\gep}(x)}\,\dd x\big) \right]=\bbE\left[e^{\alpha     \int_{\mathcal{O}}(e^{ \bi \beta  X_\epsilon(x)+\frac{\beta^2}{2}K_{\gep}(x)}+ e^{- \bi\beta  X_\epsilon(x)+\frac{\beta^2}{2}K_{\gep}(x)})\,\dd x} \right]
\\ = \sum_{n=0}^{\infty}\frac{\alpha^n}{n!}\sum_{k=0}^n\binom{n}{k}\E\Big[\Big(\int_{\mathcal{O}}e^{\mathbf{i}\beta X_\epsilon(x)}\,\dd x\Big)^{k}\Big(\int_{\mathcal{O}}e^{-\mathbf{i}\beta X_\epsilon(x)}\,\dd x\Big)^{n-k}\Big]=\mathcal{Z}^K_{\alpha,\beta,\gep}(\mathcal{O}).
\end{multline}
%

 The Coulomb/Yukawa interaction gives rise to a log-gas only in dimension $2$ because the Green function is of log-type only  in dimension 2. However,   the phenomenology governing the behavior of $\log$-gases is independent of  the dimension. Beyond its application to physically relevant $2D$-Coulomb/Yukawa cases,
the problem of giving an interpretation to the limit \eqref{2DCGapprox} using the Sine-Gordon representation \eqref{defSGintroepsilon} is of interest from a mathematical perspective because it provides a testbed for renormalization techniques, known as ultraviolet renormalization of the Sine-Gordon model (as opposed to infrared renormalization which is concerned with the large volume behaviour of the model).
 
\subsection{Renormalization and the multipole picture} \label{sub:multi}

 As soon as $\beta>0 $, the definition of the limit \eqref{defSGintroepsilon} is  nontrivial  since $X_{\gep}$ diverges pointwise.
 When $\beta^2<d$, however, for a large class of approximation schemes  the limit 
 $$\lim_{\gep \to 0}\exp\left(2\alpha  \int_{\mathcal{O}}e^{\frac{\beta^2}{2}K_{\gep}(x)}\cos(\beta X_\epsilon (x)) \,\dd x\right)$$ 
 exists and is integrable, and does not depend on the scheme we use for the approximation.
 This makes the limiting measure  \eqref{defSGintroepsilon}  absolutely continuous with respect to $\P$, as was shown in  \cite{froh76}. As can be seen directly, the limit \eqref{2DCGapprox} is positive (and finite), and induces a probability distribution   on 
 the set of charged particles on $\cO$ which only gives mass  to configurations with finitely many particles.

 When $\beta$ passes the threshold  $\sqrt{d}$, the limit \eqref{2DCGapprox} becomes infinite and $\int_{\mathcal{O}} e^{\frac{\beta^2}{2}K_{\gep}(x)}\cos(\beta X_\epsilon (x)) \, \dd x$ does not converge as a random variable. 
 When $\beta\in [\sqrt{d},\sqrt{2d})$ the divergence can be tamed by subtracting a number of field independent counter-terms in the exponential \eqref{2DCGapprox}. This corresponds to multiplying the sequence by $\exp( -\sum_{i=1}^k  p_i(\gep)\alpha^i)$, for a polynomial in $\alpha$  whose coefficients $p_i(\gep)$  diverge as $\epsilon\to 0$.
 The  number of necessary counter-terms and their values  are obtained by 
 an asymptotic analysis of the cumulants of the random variable 
\begin{equation}\label{IbouBa}
 \int_{\mathcal{O}}e^{\frac{\beta^2}{2}K_{\gep}(x)}\cos(\beta X_\epsilon (x)) \,\dd x.
 \end{equation}
 When $\beta^2<2d$ only a finite number of cumulants diverge. 
This number  increases with $\beta$ and tends to infinity when $\beta$ approaches $\sqrt{2d}$. Furthermore the analysis of  cumulants   allows us to define a  nontrivial  limit for the distribution \eqref{defSGintroepsilon} which is (conjecturally)  singular with respect to $\P$.

\medskip

When $\beta^2\ge 2d$,  infinitely many  cumulants diverge. This   makes ultraviolet renormalization impossible and there is thus no possible interpretation of \eqref{2DCGapprox} and \eqref{defSGintroepsilon} beyond  this value. The origin of these divergencies can be understood in a simple way via the multipole picture  originally described in  \cite{froh76}. Let us expose it in the framework of the $2D$-Coulomb gas\footnote{in which case the dimension is $d=2$ and KT transition occurs at $\beta^2=2d=4$.} (but the phenomenology applies in every dimension and for other potentials)

\begin{itemize}
 \item When $\gb^2<2$, the limit \eqref{2DCG} exists and corresponds to a gas of a finite number of  isolated particles.
 \item When $\beta^2\in [2,4)$, the partition function is dominated by the contribution of dipoles consisting of two nearby particles (within distance of order $\gep$ where $\gep$ is the scale at which $K$ is smoothened) with opposite charges. Each dipole  contributes to a negative energy $- |\log \gep|$, so that a configuration of $2n$ particles forming $n$ dipoles correspond to
 an energy  $- n|\log \gep|$.
Such configurations occupy a volume of the state space $\mathcal{O}^{2n}$ which is of order $\epsilon^{2n}$ (if the position of positive charges are chosen freely, the negative charges have to be located in   balls of volume $\gep^2$ around the locations of positive ones). The total contribution of isolated dipoles to the partition function \eqref{2DCG} restricted to configurations of size $2n$ is thus of order $\epsilon^{(2-\beta^2)n}$, which diverges for $\beta^2>2$ (when $\beta^2=2$ a logarithmic divergence is obtained by summing over intermediate scales between $\gep$ and the macroscopic one). The gas is therefore dominated by configurations made up of
a large number (of order $\epsilon^{(2-\beta^2)n}$) of such dipoles.
While the total number of charged particle diverges in the limit when $\gep$ goes to zero, local cancelations of charges on short distances allows us to define a limit of the charge distribution.

\item When $\beta^2\in [3,4)$, the contribution of quadrupoles (that is a combination of two closely located dipoles in a configuration that makes the resulting energy smaller than the energy of the sum of two independent dipoles), while not dominant, comes to diverge also : the total weight of configurations formed of $n$-quadrupoles being of order $\epsilon^{(6-2\beta^2)n}$ (when $n$ is even). For this reason a second diverging counter-term is required in the renormalization.
\item When $\beta^2\in[10/3,4)$ the contribution of sextupoles  has to be taken into account, and in general $2p$-poles starts having a diverging contribution to the partition function when $\gb^2\in[2(2-p^{-1}),4)$.
\item When $\beta^2\ge 4$, the contribution of $2p$-poles diverges for all $p$. Furthermore, the relative weight of $2p$-poles (considering energy and entropy) which is $\epsilon^{(2(2p-1)-p\beta^2)n}$, becomes increasing in $p$, indicating a total collapse of the system.
\end{itemize}
This scenario is rather well understood  \cite{BGN82,BK,DH,Nic,NRS}, combining log gas/Sine Gordon approaches. Yet all the methods presented so far suffer from restrictions. The main reason for this is that they rely on renormalization group (RG) techniques which,  though powerful,   are especially adapted to a translational invariant context, for which the RG map is most easily studied.  As a consequence, the papers mentioned above have considered  only the   free boundary version of the Coulomb gas or  Yukawa gas (in two dimensions) as other boundary conditions for the Coulomb gas are usually not translationally invariant. Furthermore, the existence of the correlation function of fractional charge densities could not be established using renormalization techniques,  as it requires analyzing the model under a local change of background geometry (see Subsection \eqref{sec:frac} for a more precise statement) which strongly breaks translational invariance. Let us mention that correlation functions are studied in  \cite{BK},  with a restriction on the  range of allowed parameters $\beta$, with a method that relies on Hamilton-Jacobi equations. Besides these restrictions, the methods above  either bring a limitation in the range of allowed $\beta$, or require $\alpha$ to be small, or only provide bounds for the partition function \eqref{2DCGapproxics} without establishing the existence of the limit.  Therefore, it is certainly fair to say that the global understanding of the model is far from complete from the mathematical angle. 

Finally, let us mention some further related results in dimension $d=2$. A  large deviation results  for a space discretization of \eqref{2DCG} in the case $\beta^2<d$ is obtained in \cite{LSZ}.  A dynamical approach of the model is studied in \cite{HS} for $\beta^2 \in [0,\tfrac{4}{3}d)$, and then extended to the whole subcritical regime $\beta^2\in [0,2d)$ in \cite{CH}. These papers construct the natural Langevin dynamics associated to the measure described by \eqref{2DCG}.    The Berezinsky-Kosterlitz-Thouless phase transition occurring at  $\beta^2=2d$  is studied in \cite{falco, frs}.

In this paper we revisit the Sine-Gordon model in the case $d=1$. This situation is also known as the \textit{ boundary Sine-Gordon model} in physics \cite{FZ,S1} because it arises when constraining the $2D$-Coulomb gas particles to live on the boundary of a smooth planar domain and serves as a model for the Kondo effect\footnote{Based on an exact perturbative expansion of the Anderson-Yuval reformulation of the anisotropic Kondo problem, see \cite{FS}.}, resonant tunneling in quantum wires or between quantum Hall edge states.  In this context, we are able to produce a short proof for the renormalization of the partition function of the model  and convergence of Gibbs measures for both the Sine-Gordon model and the $\log$-gas (Theorems \ref{th:main} and \ref{th:mass}). Our approach is flexible in the model (Yukawa/$\log$-gas or even any reasonable perturbation of the $\ln$ in \eqref{2DCG}) and allows us to deal with any possible value of $\alpha\in\R$ and $\beta^2<2d$. Furthermore we can also deal with renormalization of fractional charge density correlation functions (Theorem \ref{th:lecorrels}).   Our argument relies on martingale techniques: so we choose a regularization $X_\epsilon$ of our $\log$-correlated field  such that the family $\big(  \int_{\mathcal{O}}\cos(\beta X_\epsilon (x))  e^{\frac{\beta^2}{2}K_{\gep}(x)} \,\dd x \big)_\epsilon$ appearing in \eqref{defSGintroepsilon} is a martingale. All the details concerning this martingale decomposition and its basic properties are given in Section \ref{sec:setup} where our setup is introduced. The remainder of  our strategy is detailed in Section \ref{sub:orga}. As explained  below the dimensional restriction $d=1$ appears in only one crucial step of the proof. While we believe the method can be extended to higher dimensions it would require  some additional ideas.

 \section{Setup and results}\label{sec:setup}
\subsection{Definitions and assumptions}
From  now on, we restrict  the problem to the case of dimension $d=1$. We choose however to keep writing $d$ for the dimension in the statements in order to underline that large chunks of the proof can carry to the general case and and also to make more transparent the correspondence with expected/existing  results in higher dimensions. We replace $\cO$  by a bounded  interval $I\subset \R$. Our method also applies  without changes  if one replaces $I$ by a closed Jordan curve in $\bbR^2$ with distances measured by arclength.
Let $K(x,y)$ be a positive semidefinite kernel satisfying
\begin{equation}\label{genpote}
K(x,y)= \log \frac 1 {|x-y|}+ h(x,y)
\end{equation}
where $h$ is   bounded continuous on $I^2$.
The measure of integration we consider on $I$ (which is simply Lebesgue measure in  \eqref{2DCG}) is a (positive) Borel measure of the form 
\begin{equation}\label{boundedmease}
\mu(\dd x)= g(x)\dd x, \quad \quad \quad g \text{ bounded measurable.}
\end{equation}
Choosing $\mu$ of this form with $g$ bounded is not an artificial restriction since the presence of singularities can change the multipole picture presented above.
 We assume that our kernel can be written in the form 
\begin{equation} \label{sumkernel}
K(x,y)= \lim_{t\to \infty}\int^t_0 Q_u(x,y)\dd u=: \lim_{t\to \infty} K_t(x,y),
\end{equation}
where, for each $u\geq 0$, $Q_u$ is a bounded symmetric positive semidefinite kernel for which we will specify some regularity assumptions (see Assumption \ref{ass:WN} below and an example in Remark \ref{remark1}). The bounded kernel $K_t$ plays the role of $K_{\gep}$ in the introduction (with $e^{-t}$ playing the role of $\epsilon$) and we consider a log-gas partition function of the following form
\begin{equation}\label{SGKI}
\mathcal{Z}_{\alpha,\beta,t}(I)=\sum_{n= 0}^{\infty} \frac{\alpha^n}{n!}\sum_{(\gl_i)_{i=1,\dots,n}\in \{-1,1\}^n}
\int_{I^{n}}\exp\Big(- \beta^2 \sum_{1\le i<j\le n}\lambda_i\lambda_j K_t(x_i,x_j) \Big) \prod_{i=1}^{n} \mu(\dd x_i).
\end{equation}
The state space  of particle configurations is the disjoint union $\Xi_I:= \coprod_{n\ge 0} (  I^n\times \{-1,1\}^n)$, namely the set $\{(n, \mathbf{x},\lambda)\,  : \, n\in \N, \mathbf{x}\in I^n,\lambda\in \{\pm 1\}^n\}$ equipped with its canonical sigma-algebra. A measurable function $f$ on $\Xi_I$ is thus a sequence $(f(n,\cdot,\cdot))_n$ of measurable functions on $  I^n\times \{-1,1\}^n$.  The partition function \eqref{SGKI} induces a probability measure $ \bP_{\alpha,\beta,t}$ (with expectation $ \bE_{\alpha,\beta,t}$)  on $\Xi_I$  by setting for arbitrary bounded measurable function $f$ on $\Xi_I$ (below $ \mathbf{x}=(x_i)_{i=1,\dots,n}$, $\gl=(\gl_i)_{i=1,\dots,n}$ are vectors with $n$ coordinates)
\begin{equation}
 \bE_{\alpha,\beta,t}( f)=\frac{1}{\mathcal{Z}_{\alpha,\beta,t}(I)}\sum_{n= 0}^{\infty} \frac{\alpha^n}{n!}\sum_{\gl\in \{-1,1\}^n} \int_{I^n}f(n, \mathbf{x},\lambda)\exp\Big(- \beta^2 \!\!\! \sum_{1\le i<j\le n}\!\!\! \lambda_i\lambda_j K_t(x_i,x_j) \Big) \prod_{i=1}^{n} \mu(\dd x_i).
\end{equation}
Labels being irrelevant, the only physically relevant quantity under the probability  $\bP_{\alpha,\beta,t}$ is the charge distribution which we denote by $\nu$. It is obtained as the pushforward (or image measure) of the probability measure  $\bP_{\alpha,\beta,t}$ by the map   $\Pi: \Xi_I \mapsto  \cM(I)$ (the set of signed measures on $I$) which sums signed Dirac masses $\pm\gd_{x_i}$ corresponding  to particles' locations and charges
 $$\nu:=  \Pi \left(n,   \mathbf{x}, \gl \right) := \sum_{i=1}^n \gl_i \gd_{x_i}.$$
Now, we construct the Sine-Gordon representation for all values of $t\ge 0$  on the same probability space $(\Omega,\mathcal{F},\P)$
(this is important for our analysis). We consider  thus a centered Gaussian field $(X_t(x))_{t\geq 0,x\in I}$    with covariance function 
\begin{equation}\label{reprsentK}
\E[X_t(x)X_s(y)]=K_{t\wedge s}(x,y),
\end{equation}
which is almost surely continuous in both variables $t,x$ : the existence of such a process results both from the assumption \eqref{sumkernel}, which  ensures that the kernel $K$ is  positive semidefinite  on $(\bbR_+\times I)^2$, and Kolmogorov's continuity criterion, which can be applied thanks to the regularity assumptions on $Q$ made in Assumption \ref{ass:WN} below. 
Also, note that for a fixed $t$ the process $(X_t(x))_{x\in I}$   is a centered continuous Gaussian field with covariance $K_t$. We denote by $\mathcal{F}_t$ ($t\geq 0$) the filtration generated by $\{X_s(x); \,s\leq t, x\in I\}$, and by $\cF_{\infty}$ the $\sigma$- algebra generated by $\cup_{s\ge 0} \cF_s$.
 Considering the martingale
 \begin{equation}\label{def:mt}
M^{(\gb)}_t:= \int_{I} \cos(\gb X_t(x)) e^{\frac{\gb^2}{2}K_t(x,x)}\mu (\dd x)
\end{equation}
we have 
 $\bbE[e^{ \alpha M^{(\gb)}_t}]=\mathcal{Z}_{\alpha/2,\beta,t}$.
There is also a Sine-Gordon representation for the Fourier transform of the mass distribution. Considering a bounded continuous function $\theta: I \to \bbR$ (recall that $\nu$ denotes the random charge distribution under  $\bP_{\alpha,\beta,t}$)  we obtain by repeating the computation leading to \eqref{asdf}
\begin{equation}\label{scroot}
  \bE_{\alpha/2,\beta,t}[ e^{\bi \langle \theta ,  \nu\rangle}]=\bE_{\alpha/2,\beta,t}[ e^{\bi \sum_{i=1}^n \gl_i \theta (z_i)}]
 =\frac{\bbE[e^{ \alpha M^{(\gb,\theta)}_t}]}{\bbE[e^{ \alpha M^{(\gb)}_t}]},
\end{equation}
where we introduce another martingale $M^{(\gb,\theta)}$ by
\begin{equation}\label{mbetateta}
 M^{(\gb,\theta)}_t:= \int_{I} \cos(\gb X_t(x)+\theta(x)) e^{\frac{\gb^2}{2}K_t(x,x)}\mu (\dd x). 
 \end{equation}
Our aim is thus to obtain results concerning the asymptotic behavior of the Laplace transform of  $M^{(\gb,\theta)}_t$ for a large class of $\theta$ and deduce consequences concerning the charge distribution.
We obtain such results under the following regularity assumption for the covariance function (recall $d=1$).
\begin{assumption}{\bf (Smooth white noise decomposition)}  \label{ass:WN}
 \begin{enumerate} 
 \item For every $u$, $(x,y)\mapsto Q_u(x,y)$ is of class $C^2$ and there exists a constant $C$ such that for any $x,y\in I$ and $u>0$
 \begin{equation}\begin{split}\label{laderive}
    |Q_u(x,y)| &\le C  (1+e^u |x-y|)^{-(d+1)},\\
    |\partial_x Q_u(x,y)| &\le C e^{u}(1+ e^u|x-y|)^{-(d+1)},\\
       |\partial_x \partial_y Q_u(x,y)| &\le C e^{2u}(1+e^u|x-y|)^{-(d+1)}.
\end{split}
 \end{equation}

\item $\sup_{t\geq 0} \sup_{x,y\in I}|K_t(x,y)+\ln (|x-y|\vee e^{-t} )|<\infty$.

\end{enumerate}
\end{assumption}
\begin{remark}\label{remark1} The kernel $Q_u$ in \eqref{sumkernel} is meant to account for the correlations present at scale
$e^{-u}$. A prototypical example to have in mind is  $Q_u(x,y):=Q(e^ux, e^u y)$ where $Q$ is a fixed smooth translation invariant covariance function on $\bbR^d$ for which $Q(x,0)$ and its first two derivatives display fast decay. The Gaussian kernel $Q_u(x,y)= \exp(-\frac{(e^{u}|x-y|)^2}{2})$ which perfectly fits this framework is of peculiar importance as it is the basis for the construction of $2D$ Gaussian Free Fields (including Dirichlet/Neumann GFF or the massive GFF appearing in the Sine-Gordon representation of the Yukawa gas), which we can  restrict on a one-dimensional manifold.
\end{remark}
\begin{remark}
Note that Assumption \ref{ass:WN} implies in particular that $K:=\lim_{t\to\infty}K_t$ exists and is of the form given in Equation \eqref{genpote}.
\end{remark}
We let $(\mathcal{C}^{(\beta,\theta)}_k(t))_k$ denote the successive cumulants  of the martingale $M^{(\beta,\theta)}_t$ (a reminder of the definition of cumulants is given in Section \ref{sec:cum}) and set $\mathcal{C}^{(\beta,\theta=0)}_k(t)=\mathcal{C}^{(\beta)}_k(t) $. 
In the Sine-Gordon representation, the multipole picture presented in Section \ref{sub:multi} corresponds to an explosion of even order  cumulants of the martingale when $t\to\infty$. 
This phenomenon is described by the following sequence of successive   thresholds $(\beta_n)_{n\in\N^*}$ of the Sine-Gordon model (recall that here $d=1$) 
\begin{equation}\label{seuil}
\beta_0:=0,\quad \forall n \ge 1, \quad \quad \beta_n:= \sqrt{2d(1-\tfrac{1}{2n})}.
\end{equation}
 For  $\beta \in [\beta_{n-1},\beta_{n})$, $\alpha\in \bbR$, and $f$ a bounded ($\cF_{\infty}$-)measurable function, we define  the renormalized partition function which integrates $f$ as
 \begin{equation}\label{th:cv}
\bar{\mathcal{Z}}_{\alpha/2,\gb,t}(f)=:\E[f(X)e^{\alpha M^{(\gb)}_t}]e^{-\sum_{i=1}^{n-1}\frac{\alpha^{2i}}{(2i)!} \mathcal{C}^{(\beta)}_{2i}(t)}.
\end{equation}

\subsection{Renormalization for the log-gas and its Sine-Gordon representation}

Our first result shows that the renormalized partition functions converge,
and that the tilt by $M_t$ induces a (non-Gaussian) limiting measure for the process 
$X_t$ when $t$ tends to infinity. Recall that our results hold in dimension $d=1$.
 \begin{theorem}\label{th:main}
Assume $\beta^2<2d$ and $\alpha\in\R$. If $f$ is bounded and measurable with respect to 
$\cF_{t_0}$ for some $t_0\in (0,\infty)$, then the following  limit is well defined
$$\mathcal{Z}_{\alpha,\gb}(f):=\lim_{t\to\infty}\bar{\mathcal{Z}}_{\alpha,\gb,t}(f).$$
Furthermore, the mapping
$$f\mapsto   \bE_{\alpha,\beta}(f):= \frac{\mathcal{Z}_{\alpha,\beta}(f)}{\mathcal{Z}_{\alpha,\beta}(1)} $$ 
can be extended  to all bounded ($\cF_{\infty}$-)measurable functions $f$ and defines a probability measure on  $\cF_{\infty}$, under which the process $t\mapsto X_t(x)$ is almost surely continuous, for all $x\in I$.
\end{theorem}

\begin{remark}\label{oddcum}
In \eqref{th:cv}, the sum $\sum_{i=1}^{n-1}\frac{\alpha^{2i}}{(2i)!} \mathcal{C}^{(\beta)}_{2i}(t)$ in the exponential factor represents the diverging terms that have to be subtracted to $\E[f(X) e^{\alpha M_t}]$ in order to get a converging expression. We will show below that for $\beta^2<2d$ only cumulants of even order might diverge when $t\to \infty$ while those of odd order converge.
The sum can hence be replaced by  $e^{\sum_{i=1}^{2n-1}\frac{\alpha^i}{i!} \mathcal{C}^{(\beta)}_{i}(t)}$ without modifying the result  (besides a modification of the value of the limit $\mathcal{Z}_{\alpha,\gb}(f)$ by a factor that does not depend on $f$ but this is irrelevant).
\end{remark}

\begin{remark}
 When $\gb^2<\beta_{2}$, the convergence result  \eqref{th:cv} is an instance of mod-Gaussian convergence for the variable $M^{(\gb)}_t$ (when $t$ goes to infinity) in the terminology of \cite{JKN}. Indeed recall that a sequence of real-valued random variables $(Z_n)_n$ is said to converge in the mod-Gaussian sense if we can find two sequences $(m_n)_n\in \R^\N$ and $(\gamma_n)_n\in(\R_+)^\N$ and a complex-valued function $\Phi$ continuous at $0$ such that
 $$\forall u\in\R,\quad \lim_{n\to\infty}e^{-iu  m_n+u^2\gamma_n/2}\E[e^{iuZ_n}]=\Phi(u).$$
This notion of convergence is useful to  extract the interesting  behaviour of some sequence of random variables beyond ``coarse diverging  Gaussian contributions". In the Sine-Gordon model this notion is relevant for $\beta_1\leq \beta<\beta_2$ (the convergence being standard for $0\leq \beta<\beta_1$). 
When $\beta \geq  \beta_{2}$, the convergence result  \eqref{th:cv} generalizes the notion of mod-$\phi$ convergence exposed in \cite{FMN} as it requires subtracting higher order Taylor expansion terms  in $u$.
\end{remark}

\noindent While this is not an immediate consequence of the result,
 the proof of Theorem \ref{th:main} provides also the key elements to establish the convergence in law of the charge distribution in the regime $\gb^2<2d$.

\begin{theorem}\label{th:mass}
 When $\gb\in (0,\sqrt{2d})$, for any $1/2$-H\"older continuous function $\theta$, the following limit exists
 \begin{equation}
  \lim_{t\to \infty } \bE_{\alpha,\beta,t}[ e^{\bi \langle \theta ,  \nu\rangle}]= :\Psi(\theta)
 \end{equation}
Moreover $\Psi$ is continuous for the H\"older norm

$$\|\theta\|_{1/2}:= \max_{x\in I} |\theta(x)|+ \max_{x,y\in I^2} \frac{|\theta(x)-\theta(y)|}{\sqrt{|x-y|}}.$$
In particular, the charge distribution $\nu$ converges in law under $\bP_{\alpha,\beta,t}$ as $t\to\infty$ in the Schwartz space of tempered distributions\footnote{We consider here the space of tempered distributions as the topological dual of the Schwartz space of smooth functions with fast decay at infinity equipped with the usual semi-norms. We do not recall this here but the reader can refer to \cite{bierme} for a reminder of the topological setup as well as  a proof of the L\'evy continuity theorem on the space of tempered distributions.}.
 \end{theorem}

The above convergence result for the charge distribution $\nu$ towards a non degenerate limit occurs for $\beta^2\in[d,2d)$ in spite of the fact that the total number of particles tends to infinity. This is due to the dipole picture mentioned in the introduction. While an infinite number of charged particles are present in the log-gas, short range cancelations makes the quantity $\langle \theta ,  \nu\rangle$ well defined for $\theta$ sufficiently regular.

\subsection{Renormalization for correlation functions}\label{sec:frac}

Our last result concerns the asymptotics for the correlation functions associated with fractional charges.
The aim is, given $m\in\N^*$, $\boldsymbol{\eta}=(\eta_1,\dots, \eta_m)\in \bbR^m$ and $\mathbf{z}=(z_1,\dots,z_m)\in I^m$ distinct (i.e. $z_i\not=z_j$ for $i\not=j$), to determine the asymptotics of the partition function of a system, where charges $\eta_i$ have been placed on the sites $z_i$ for $i\in \lint 1, m \rint$, defined by
\begin{equation}\label{zhatcorrel}
 \hat{\mathcal{Z}}_{\alpha,\beta,t}({\bf z},\boldsymbol{\eta}):=\sum_{n= 0}^{\infty} \frac{\alpha^n}{n!}\sum_{\gl\in \{-1,1\}^n}
\int_{I^{n}}e^{- \beta^2 \left(\sum_{1\le i<j\le n}\lambda_i\lambda_i K_t(x_i,x_j)+ \sum_{i=1}^n \sum_{l=1}^m \gl_i \eta_l  K_t(x_i,z_l) \right)}\prod_{i=1}^{n}\mu(\dd x_i).
\end{equation}
More precisely, we are interested in the asymptotics of (recall \eqref{SGKI}) $\hat{\mathcal{Z}}_{\alpha,\beta,t}({\bf z},\boldsymbol{\eta})/\mathcal{Z}_{\alpha,\beta,t}(I)$. Also, we will see that the Sine-Gordon representation of $ \hat{\mathcal{Z}}_{\alpha,\beta,t}({\bf z},\boldsymbol{\eta})$ yields
\begin{equation} \label{espcorrel}
  \hat{\mathcal{Z}}_{\alpha/2,\beta,t}({\bf z},\boldsymbol{\eta})= \bbE\left[e^{\sum_{l=1}^m \left(\textbf{i}\beta \eta_l X_t(z_l)+ \frac{\gb^2 \eta^2_l}{2}K_t(z_l,z_l)\right)} e^{\alpha  M^{(\beta)}_t }\right],
\end{equation}  
 with $M^{(\beta)}_t$ defined by \eqref{def:mt}, which corresponds to the standard definition of correlation functions in the Sine-Gordon model.

\begin{theorem}\label{th:lecorrels}
Setting $\|\eta\|_{\infty}:=\max|\eta_i|$,   if $(1+2\|\eta\|_{\infty})\beta^2< 2$, 
then we have, for all $\alpha\in \bbR$ and all set of $m$ distinct points
\begin{equation}
\langle \prod_{l=1}^me^{i\beta\eta_lX(z_l)}\rangle :=\lim_{t\to \infty}  \frac{\hat{\mathcal{Z}}_{\alpha,\beta,t}({\bf z},\boldsymbol{\eta})}{\mathcal{Z}_{\alpha,\beta,t}(I)}
\end{equation}
exists and is continuous in $\boldsymbol{\eta}$ and ${\bf z}$.
 
\end{theorem}

\begin{rem}
 We expect that the result above remains valid under the less strict  (and likely optimal)  assumption
 $\beta^2 \max(\|\eta\|_{\infty}, 1/2)<1$. However establishing such a result would require some refinement of our technique, which we leave for future work.  
\end{rem}

\subsection{Organization of the remainder of the paper} \label{sub:orga}

 Let us provide a few details concerning the strategy of our proof for the three results displayed above.  In Section \ref{sec:cum}, we present a new probabilistic representation of  the cumulant generating function of a random variable on a filtered probability space,  valid with a large extent of generality and of independent interest\footnote{Note added in proof: the formula relating cumulants with martingale brackets (Lemma \ref{lem:cum}) has recently found a large variety of  applications, including   finance: see \cite{FGR}.}. 
  If one pushes the expansion to $n$-th order, the first $n$ cumulants are given in terms of expectations of brackets of   recursively defined martingales.  This expansion does not present any technical difficulty. 
The technical part of the paper  is Section \ref{sec:proofmain}, in which the above mentioned martingale brackets are computed for the random variable \eqref{IbouBa} (for the martingale approximation setup).
 The computation of the brackets \textit{per se} only requires standard Itô calculus, but it produces multivariate  integrals involving a high number of indices. Our main contribution here is to find a systematic way to control these integrals (see Proposition \ref{prop:cumul}) using a refinement of the electrostatic inequalities (Lemma \ref{simplifbis}) introduced by Onsager \cite{cf:Ons, cf:Jeanpierre} (originally proved for the $3D$ Coulomb potential  but the inequality can also be extended to other positive definite potentials). This is the step where $d=1$ is required: while our electrostatic inequality is valid in any dimension, some technicalities limit the remainder of the argument to the one-dimensional case.

 In Sections \ref{sec:partition}, \ref{sec:charge} and \ref{sec:correl}, we apply and adapt the technical results of Sections \ref{sec:cum} and  \ref{sec:proofmain} to prove our main results concerning the convergence of the Sine-Gordon partition function (Theorem \ref{th:main}), the charge distribution for the $\log$-gas (Theorem \ref{th:mass}) and the  Sine-Gordon correlation functions (Theorem \ref{th:lecorrels}) respectively.

\section{Cumulants of continuous martingales, a general approach}\label{sec:cum}

 In this section we provide a general scheme which allows us to compute the successive cumulants of an arbitrary continuous martingale, or equivalently of any random variable defined on  some probability space equipped with a continuous filtration.  Let us start by recalling the definition of the cumulants of a random variable. If $Z$ is a random variable such that $\bbE[e^{\gep |Z|}]<\infty$ for some $\epsilon>0$, then the $\log$-Laplace transform of $Z$, $\alpha \mapsto \ln\E[ \exp\left(\alpha Z \right) ]$ is analytic in a neighborhood of zero and thus admits a power series expansion of the following form 
 \begin{equation}\label{cumul}
 \ln \bbE[ \exp\left(\alpha Z \right)  ] =\sum_{i=1}^{\infty} \frac{\cC_i(Z)}{i!} \alpha^i.
\end{equation}
 The expression of the coefficients $\cC_i(Z)$ can be obtained by an extensive use of the Taylor formula for $x\mapsto\ln(1+x)$. In particular for $i\ge 1$
\begin{equation}\label{defcumulant}
 \cC_i(Z):= \bbE[ Z^i]+q_{i}\left( \bbE[Z], \bbE[Z^2], \dots, \bbE[Z^{i-1}]\right).
\end{equation}
 where $q_i$ is a polynomial in $i-1$ variables. In full generality, Equation \eqref{defcumulant} defines the $i$-th cumulant of $Z$ under the less strict requirement $\bbE[|Z|^i]<\infty$.

 \medskip
 
 We consider $(M_t)_{t\ge 0}$ a continuous martingale with respect to a filtration $(\cF_t)_{t\ge 0}$ which starts from $\cF_0=\{0,\gO\}$ (this last assumption ensures that $M_0$ is almost surely constant).
We define inductively a sequence of processes $A^{(i)}_t$ and $(M^{(i,t)}_s)_{s\in [0,t]}$ as follows.
First set
\begin{equation}\label{defa1}
A^{(1)}_t=M_t \quad \text{ and } \quad (M^{(1,t)}_s)_{s\in[0,t]}:= \bbE[A^{(1)}_t \ | \cF_s]-\bbE[A^{(1)}_t]=M_s-M_0.
\end{equation}
Then for $i\ge 2$ we define $A^{(i)}$ in terms of the 
quadratic variations of previous order martingales provided that they are well defined
\begin{equation}\label{dafai}
A^{(i)}_t:=  \frac{1}{2}\sum_{j=1}^{i-1} \langle M^{(j,t)},  M^{(i-j,t)}\rangle_t \quad \text{ and } \quad 
 M^{(i,t)}_s:=\bbE[A^{(i)}_t \ | \cF_s]-\bbE[A^{(i)}_t].
\end{equation}
While our result might hold with greater generality,  we assume for simplicity (and because this corresponds to the applications we have in mind) that 
all the quantities above are well defined and
the quadratic variations above are essentially bounded in the sense that 
for every $i$ and $t$  
\begin{equation}\label{bounded}
\| \langle M^{(i,t)} \rangle_t\|_{\infty}<\infty \quad \text{ where } \quad  \| Z\|_{\infty}:=\inf\{ u\ge 0  : \bbP[|Z|>u]=0\}. 
\end{equation}

\begin{lemma}\label{lem:cum}
Given $j \ge 1$, consider the martingale $(N^{(j,\alpha,t)}_s)_{s\in[0,t]}$ defined by
 \begin{equation}\label{defn}
 N^{(j,\alpha,t)}_s:=\sum_{i=1}^j \alpha^i M^{(i,t)}_s.
 \end{equation}
Under the assumption \eqref{bounded}, we have the decomposition
\begin{equation}\label{cumdecomp}
 \log \bbE\left[e^{\alpha M_t}\right]=\sum_{i=1}^j \alpha^i \bbE[A^{(i)}_t]+\log
 \bbE\left[e^{ N^{(j,\alpha,t)}_t-\frac{1}{2}\langle N^{(j,\alpha,t)}\rangle_t}e^{Q^{(j,\alpha)}_t}\right].
\end{equation}
where
\begin{equation} \label{defq}
Q^{(j,\alpha)}_t:= \frac{1}{2}\sum_{i=j+1}^{2j} \alpha^{i} 
  \sum_{k=i-j}^{j}\langle M^{(k,t)},M^{(i-k,t)} \rangle_t .
  \end{equation}
As a consequence,   the $i$-th cumulant of $M_t$ is given by 
\begin{equation}
 \cC_i(M_t):= i! \bbE[A^{(i)}_t].
\end{equation}
\end{lemma}

\begin{proof}
To illustrate the idea of the proof let us start with the case $j=1$ (which of course could be obtained with a simpler direct computation). We have 
\begin{equation}
 \bbE\left[e^{\alpha M_t}\right]= e^{\alpha \bbE[M_t]}  \bbE\left[e^{\alpha M^{(1,t)}_t}\right]=e^{\alpha \bbE[M_t]}  \bbE\left[e^{\alpha M^{(1,t)}_t-\frac{\alpha^2}{2}\langle M^{(1,t)}\rangle_t}e^{\frac{\alpha^2}{2} \langle M^{(1,t)}\rangle_t}\right].
\end{equation}
Considering $$(\dd \tilde \bbP/\dd \bbP)(\go):=e^{\alpha M^{(1,t)}_t-\frac{\alpha^2}{2}\langle M^{(1,t)}\rangle_t}$$ as probability density we have thus  
\begin{equation}
 \ln  \bbE\left[e^{\alpha M_t}\right]= \alpha \bbE[M_t]+
 \ln\tilde \bbE\left[ e^{\frac{\alpha^2}{2} \langle M^{(1,t)}\rangle_t} \right]. 
\end{equation}
Using our assumption \eqref{bounded}, we obtain that the last term is at most of order $\alpha^2$ and we can conclude that the first cumulant is given by $ \bbE[M_t]$. Our construction of $A^{(i)}$ and $M^{(i,t)}$ has been made so that we can iterate the above process, each $A^{(i)}_t$ being designed to cancel the quadratic variations of terms that have appeared on previous steps. 
    The computation  leading to \eqref{cumdecomp} is the following
 \begin{equation}\begin{split}\label{larecre}
  \langle N^{(j,\alpha,t)} \rangle_t&= \sum_{i=2}^{2j} \alpha^{i} 
  \sum_{k=1}^{i-1}\ind_{\{ \max(k,i-k)\le j\}} \langle M^{(k,t)},M^{(i-k,t)} \rangle_t\\
  &= 2 \sum_{i=2}^{j} \alpha^i A^{(i)}_t+ \sum_{i=j+1}^{2j} \alpha^{i} 
   \sum_{k=i-j}^{j} \langle M^{(k,t)},M^{(i-k,t)} \rangle_t\\
  &= 2 \left[(N^{(j,\alpha,t)}_t-\alpha M_t)+  \sum_{i=1}^{j} \alpha^i \bbE[A^{(i)}_t]  
  + Q^{(j,\alpha)}_t\right].
  \end{split}
 \end{equation}
Interpreting $e^{ N^{(j,\alpha,t)}_t-\frac{1}{2}\langle N^{(j,\alpha,t)} \rangle_t  }$ as a probability density, we can deduce 
from \eqref{larecre} that (recall  \eqref{bounded})
\begin{equation}
  \left|\log \bbE\left[e^{\alpha M_t}\right]-\sum_{i=1}^j \alpha^i \bbE[A^{(i)}_t]\right|
  \le \|Q^{(j,\alpha)}_t\|_{\infty}\le j^2 \alpha^{j+1} \max_{i\le j}\| \langle M^{(i,t)}\rangle_t \|_{\infty}. 
\end{equation}
We conclude using the characterization  of the cumulants \eqref{cumul}.
\end{proof}

\section{Iterative computation of the martingale brackets}\label{sec:proofmain}

Recalling Equation \eqref{th:cv}, an important step for the convergence in the case $F\equiv 1$ is   to show  that for every $n\ge 1$, we have $\lim_{t\to \infty}\cC^{(\gb)}_{2n}(t)$ exists and is finite for $\beta<\beta_{n}$ and that 
$\lim_{t\to \infty}\cC_{2n-1}(t)$ converges for all $\gb < \sqrt{2d}$. Hence our main effort in the proof will consist in proving convergence of these cumulants.
The convergence for general $F$ then mostly follows from the technique developed to control the cumulants.

\subsection{Computing the cumulants}\label{sub:comput}

In this  section, we present our main technical results concerning the cumulants (Proposition \ref{prop:cumul}) which allows us to ensure their convergence.
We use an inductive approach to find an integral expression for cumulants of the martingale $M^{(\gb)}_t$ \eqref{def:mt} which we denote by $\mathcal{C}_i^{(\beta)}(t)$.

\medskip

To motivate this approach, we compute the first cumulant using the procedure proposed in  Section \ref{sec:cum}.
For practical purposes we state a general formula for the quadratic variation established with basic Itô calculus.
We use it repeatedly throughout our computations. For bounded measurable functions $f$ and $g$ we have
\begin{multline}\label{correlsnh}
\int_{t_1}^{t_2}\int_{I^2}f(X_u(x),u)g(X_u(y),u)\dd \langle  X(x),  X(y) \rangle_u \mu(\dd x)\mu( \dd y) \\
= \int_{t_1}^{t_2} \int_{I^2}f(X_u(x),u)g(X_u(y),u) Q_u(x,y) \mu(\dd x)\mu( \dd y) \dd u.
\end{multline}
Within computations we often omit the dependence in $\beta$ in some notations for the sake of  readability. We set  first $M^{(1,t)}_s=M_s-\mu(I)$ for $s\leq t$. Then It\^o's formula gives
\begin{equation}
\dd M_s= -\int_{I}    \Big( \gb  \sin (\gb X_s(x)) e^{\frac{\gb^2  }{2}K_s(x,x)}\dd X_s(x) \Big)\mu(\dd x).
\end{equation}
By transforming the product of  sines as follows 
$$2 \sin(\gb X_1)\sin(\gb X_2)= \cos \left(\gb (X_1-X_2)\right)- \cos \left(\gb( X_1+X_2)\right),$$
The definition in \eqref{defa1} and Equation \eqref{correlsnh} yield
\begin{multline}\label{compute}
 A^{(2)}_t= \frac{\gb^2}{4} \int^t_0  \int_{I^2} Q_u(x_1,x_2) e^{\frac{\gb^2}{2}(K_u(x_1,x_1)+K_u(x_2,x_2))} \\
 \times \left[ \cos \left(\gb (X_u(x_1)-X_u(x_2))\right)- \cos \left(\gb( X_u(x_1)+X_u(x_2))\right) \right]\mu(\dd x_1)\mu( \dd x_2)\dd u.
\end{multline}
Computing explicitely $A^{(3)}$  is already quite involved (the explicite computation is performed in the appendix to illustrate this fact) and the number of terms for |$i\ge 4$ grows very rapidly even after taking into account the symmetries. In order to attack this problem we are going to look for a generic non explicit expression for   $A^{(i)}_t$ and try to control its growth inductively. Equation \eqref{compute} and iterating suggests that $A^{(i)}_t$ can in general be written in the following form  
\begin{equation}\label{formz}
 A^{(i)}_t=\sum_{p=0}^{\lfloor i/2 \rfloor} \int^t_0 \int_{I^i} F^{(i,p,\gb)}(\bx,u,t)e^{\frac{\gb^2}{2}\sum_{k=1}^iK_u(x_k,x_k)}
 \cos\left[ \gb\left(\sum_{k=1}^{i} \lambda^{(p)}_k X_u(x_k)  \right) \right]
 \mu(\dd \bx) \dd u.
\end{equation}
where we  use the shorthand notation  $\mathbf{x}=(x_1,\dots,x_i)$ and $\mu (\dd \mathbf{x})=\prod_{k=1}^i\mu(dx_k)$. The sequence $(\lambda^{(p)}_k)_k$ is defined by 
\begin{equation}\label{recoil}
 \lambda^{(p)}_k:=-1 \text{ for } k\le p \quad \text{ and }\quad \lambda^{(p)}_k:=1 \text{ for  } k\ge p+1,
\end{equation}
and  ${\bf x}\mapsto F^{(i,p,\gb)}({\bf x},u,t)$ is a continuous function  on $I^i$. We are going to prove \eqref{formz} by induction,
but the most important part of our task is to establish relevant properties for $F^{(i,p,\beta)}$ which imply convergence of the cumulants.
We let $D_i:= \{ \bx\in I^i : \exists k\ne l,  x_k=x_l\}$ denote the set of on-diagonal points in $I^i$.

\begin{proposition}\label{prop:cumul}
For every $i\ge 2$ and $p\leq \lfloor i/2 \rfloor$, $A^{(i)}_t$ can be written in the form \eqref{formz}, where  $F^{(i,p,\gb)}({\bf x},u,t)$ is analytic in $\beta$ and  satisfies:
\begin{enumerate}
  \item  For every ${\bf x}\in I^i\setminus D_i$  there exists a positive constant $C(\bx,\gep)$ 
  such that for every $0<u<t$, $\beta \le \sqrt{2}-\gep$, we have  
  \begin{equation}\label{didness}
  |F^{(i,p,\gb)}({\bf x},u,t)|\le C(\bx,\gep)e^{-u}.
  \end{equation}
The following convergence occurs uniformly on any compact subset of
 $I^i\setminus D_i$ and uniformly in  $\beta \in [0, \sqrt{2}-\gep]$ 
  \begin{equation}\label{limit}
  \lim_{t\to \infty} F^{(i,p,\gb)}({\bf x},u,t)=: \bar F^{(i,p,\gb)}({\bf x},u),
   \end{equation}
   and given $\bx$, the convergence is uniform in $\beta$.
\item  For every $\epsilon>0$, there exists a constant $C=C(i,\gep )>0$ such that for every $u\ge 0$, for any fixed $k$ and
 $x_k\in I$ and 
\begin{equation}\label{propz}
 \int_{I^{i-1}}\sup_{t\geq u}\sup_{\gb\in [0,\sqrt{2}-\gep]}|  F^{(i,p,\beta)}(\bx,u,t)| \mu(\dd \bx^{(k)}) \le Ce^{-(i-1)  u},
\end{equation}
where $\mu(\dd \bx^{(k)}):=\mu(\dd x_1)\dots \mu(\dd x_{k-1}) \mu(\dd x_{k+1}) \dots \mu(\dd x_i)$.

\item  
There exists a martingale  $(\bar M^{(i)}_s)_{s\ge 0}$ satisfying   $\bar M^{(i)}_0=0$
 and for every $s\ge 0$, $\| \bar M^{(i)}_s\|_{\infty}<\infty$ and  $\|\langle \bar M^{(i)}\rangle_s\|_{\infty}<\infty$ (here $\|Z\|_{\infty}=\esssup |Z|$)
such
that for any $s\ge 0$, we have almost surely 
\begin{equation}
 \lim_{t\to \infty}\| M^{(i,t)}_s-\bar M^{(i)}_s\|_{\infty}=0\quad \text{ and } \quad  \lim_{t\to \infty}\|\langle M^{(i,t)}-\bar M^{(i)}\rangle_s\|_{\infty}=0.
\end{equation} 

\end{enumerate}
\end{proposition}
\begin{remark}
 To make the proof lighter, we do not provide details concerning the uniformity in $\beta\in[0,\sqrt{2-\gep}]$ for the estimates \eqref{didness} to \eqref{propz}. They can be checked via a tedious verification procedure which does not  present any particular technical difficulty.
\end{remark}
The proof the above result (in Section \ref{sub:proof}), 
relies on a technical inequality for our potential function which 
we present now.
This is in the spirit of Onsager's electrostatic inequality 
 (also known as Onsager's Lemma \cite{cf:Ons, cf:Jeanpierre}, see also \cite[Proposition 3.9]{cf:JSW}). The main upgrade when comparing to the previously mentioned inequality, is that we provide a lower bound, not in terms of the minimal inter-particle distance, but in terms of a transport distance between the $+$ and $-$ charges.
For   $\mathbf{x}=(x_1,\dots,x_i)\in I^i$ and $u\geq 0$  and recalling  \eqref{recoil},  we define   $U_{i,p}(\bx,u)$ to the be regularized electric  potential corresponding to the set of charges placed at positions $\bx$. 
 For $s<u$, we define $U_{i,p}(\bx,s,u)$ to be the increment of  $U_{i,p}(\bx,\cdot)$ between $s$ and $u$, that is 
 \begin{equation}\begin{split}\label{minusV}
 U_{i,p}(\bx,u)&:= 2\sum_{1\le k< l \le i} \gl^{(p)}_k \gl^{(p)}_l  \int^u_0  Q_r(x_k,x_l) \dd r\, ,\\
U_{i,p}(\bx,s,u)&:= 2\sum_{1\le k< l \le i} \gl^{(p)}_k \gl^{(p)}_l \int^u_s Q_r(x_k,x_l) \dd r \, .
\end{split}
 \end{equation}
 We aim to establish  a lower bounds for $ U_{i,p}(\bx,s,u)$. 
 Our starting point is the observation that, apart from the missing diagonal terms, $U_{i,p}(\bx,s,u)$ corresponds to the variance  of a random variable, and therefore
 \begin{equation}\label{basicao}
 U_{i,p}(\bx,s,u)\ge -\sum_{k=1}^i \int_{u}^s Q_u(x_k,x_k)\,\dd u  \ge -i (u-s)-C .
 \end{equation}
 To obtain a better bound, we need to obtain 
 a better estimate  on  the above mentioned variance instead of simply relying on positivity.
 In the case $i=2p$,  our bound involves the quantity $m({\bf x})$ defined as the $1$-Wasserstein distance between the distribution of positive charges  $\sum_{k=1}^p \delta_{x_k}$  and that of the negative ones 
$\sum_{l=p+1}^{2p} \delta_{x_l}$.  It is  given by  
\begin{equation}\label{wasse}
m({\bf x})=\min_{\sigma\in \cS_p} \sum_{k=1}^p  |x_k-x_{p+\sigma(k)}|
\end{equation}
 where $\cS_p$ is the symmetric group on $p$ elements. The following result is valid in every dimension $d\ge 1$ for any $Q$ satisfying 
 \eqref{laderive}. 

 \begin{lemma}\label{simplifbis} 
For any choice of integers $i\ge 2$ and $p\le i/2$ 
there exists a constant $C$ only depending on $i$ such that for every $\bx\in I^i$:
\begin{enumerate}
\item  If $i$ is even and $p=i/2$
 \begin{equation}\label{secondestbis}
U_{i,p}(\bx,s,u)\ge -(i-1)(u-s)  -(-\log m({\bf x})-s)_+- C.
 \end{equation}
\item   In all other cases  
 \begin{equation}\label{goodzbis}
U_{i,p}(\bx,s,u) \ge -(i-1)(u-s)
 -C.
 \end{equation}
 \end{enumerate}
\end{lemma}
 The proof of this lemma is detailed in Section \ref{sub:onsager} while Proposition \ref{prop:cumul} is proved in Section \ref{sub:proof}.
We end up this section by presenting a consequence of Proposition \ref{prop:cumul} 
and Lemma \ref{simplifbis}. 

\begin{corollary}\label{coro:odd}
For each $t>0$ and $i\in\N$, we have (recall that $d=1$)
\begin{itemize}
 \item $\lim_{t \to \infty }\mathcal{C}^{(\beta)}_{2i-1}(t)=: \cC^{(\beta,\infty)}_{2i-1},$
 uniformly on all compact subsets of $[0,\sqrt{2d})$.
 \item $\lim_{t \to \infty }\mathcal{C}^{(\beta)}_{2i}(t)=: \cC^{(\beta,\infty)}_{2i},$ uniformly on all compact subsets of $[0,\beta_{i})$.
\end{itemize}
The limiting function are continuous in $\beta$.
\end{corollary}

\begin{proof} Consider $n\in\N$  and let us recall Lemma \ref{lem:cum} which gives the relation  $\mathcal{C}_n(t)=n! \E[A^{(n)}_t]$.   We have from \eqref{formz}  (recalling \eqref{minusV}) 
\begin{equation}
\bbE[A^{(n)}_t]=\sum_{p=0}^{\lfloor n/2 \rfloor} \int^t_0 \int_{I^n} F^{(n,p,\gb)}(\bx,u,t)
 e^{-\frac{\gb^2}{2}U_{n,p}(\bx,u)}
 \mu(\dd \bx) \dd u.
\end{equation}
We will now compute the limit $t\to\infty$ of the above quantity by showing for fixed $n$ that
\begin{equation}
\lim_{t \to \infty}\int^t_0 \int_{I^n} F^{(n,p,\beta)}(\bx,u,t)
 e^{-\frac{\gb^2}{2}U_{n,p}(\bx,u)}
 \mu(\dd \bx) \dd u = \int^\infty_0 \int_{I^n} \bar F^{(n,p,\beta)}(\bx,u)
 e^{-\frac{\gb^2}{2}U_{n,p}(\bx,u)}
 \mu(\dd \bx) \dd u
\end{equation}
for every $p\leq \lfloor n/2 \rfloor$ and  for appropriate values of $\beta$. Our argument depends on the situations $p\not=n/2$ or $p=n/2$, this latter situation being possible only for even $n$: 

$\bullet$ Assume $n$ is even, i.e. $n=2i$ for some $i\in\N$,  and $p=i$: \\
In that case, by    using \eqref{propz} and dominated convergence, it is sufficient to show that 
\begin{equation}
 \int^{\infty}_0  e^{-(2i-1)  u}
 \sup_{\bx\in I^{2i} }e^{-\frac{\gb^2}{2}U_{2i,p}(\bx,u)}\dd u <\infty.
\end{equation} 
For this, recall first that  \eqref{basicao}  reads $U_{2i,p}(\bx,u)\ge -2i u-C$ so that the point is then to determine whether the integral $\int^{\infty}_0  e^{-(2i-1)  u}
 e^{\frac{\gb^2}{2}2i u}\dd u$, which obviously leads to the following constraint for $\beta$: $\beta^2< 2\left(1-\frac{1}{2i}\right) =\beta^2_i$ (recall \eqref{seuil}).

$\bullet$ Assume $p\not=n/2$: \\
Note that this case happens automatically whenever $n=2i+1$ is odd, but also when $n$ is even with different number of $\pm$\,-charges. The argument is the same except that we now use  \eqref{secondestbis} in  Lemma \ref{simplifbis}, which reads $U_{n,p}(\bx,u)\ge  -(n-1)u-C$, producing the bound $\int^{\infty}_0  e^{-(n-1)  u}
 e^{\frac{\gb^2}{2}(n-1)u}\dd u$ and the constraint $\beta^2< 2$.
  
\medskip 
In both cases, the continuity in $\beta$ follows from our uniform domination estimates.
\end{proof}

\subsection{ Proof of Lemma \ref{simplifbis} (the Onsager-type inequality)}\label{sub:onsager}
 Recall that in this proof (and only in this proof) the dimension is not restricted to $d=1$. 
Let us first show that we can reduce the proof to the case
where $Q_u(x,y)=Q(e^u|x-y|)$ with $Q$ a $C^2$ positive definite function with support  in the Euclidean ball of radius one centered at the origin $B(0,1)$ and such that $Q(0)=
1$. Indeed we have
 \begin{equation}\label{invara}
 \sup_{t\geq 0} \sup_{x,y\in I}\big|\int_0^t Q(e^u|x-y|)\,du-\ln (|x-y|\vee e^{-t})|<+\infty.
\end{equation}
Therefore, Assumption \ref{ass:WN} Item (2) entails that 
 \begin{equation}
 \sup_{s,t\geq 0} \sup_{x,y\in I}\Big|\int_s^t \big(Q_u(x,y)-Q(e^u|x-y|)\big)\,\dd u\Big|<+\infty,
\end{equation}
and hence that for some constant $C>0$ we have
\begin{equation}
\sum\limits_{1\le k<l\le i} \gep_k \gep_l \int^u_{ s} Q_r(x_k,x_l)\dd r  \geq \sum\limits_{1\le k<l\le i} \gep_k \gep_l \int^u_{ s} Q(e^r|x_k-x_l|)\dd r- i(i-1) C.
\end{equation} 
Then considering $Z^{(r)}$ a centered Gaussian field of covariance $Q(e^r|x-y|)$  and $Y_r:=\sum_{k=1}^i \gep_k Z^{(r)}(x_k)$ we have
$$\Var (Y^{(r)})=i+2\sum_{1\le k<l\le i}  \gep_k \gep_l  Q(e^r|x_k-x_l|).$$
Thus we have in any case
\begin{equation}\label{slot}
2\sum_{1\le k<l\le i} \gep_k \gep_l  Q(e^r|x_k-x_l|)
\ge -i.
\end{equation}
In order to obtain a better estimate, we need a good lower bound on the variance of $Y^{(r)}$. We need to show that there exists $C>0$ such that
\begin{equation}\label{aprouve}
\int^u_s \Var (Y^{(r)})\,dr\ge \begin{cases} (u-s)- C  \quad \quad \quad&\text{when} \sum \gep_k\ne 0,\\
(u-s)- (-\log m({\bf x})-s)_+-C \quad &\text{when} \sum \gep_k= 0.
\end{cases}
\end{equation}
Let us partition $\lint 1,i\rint$ in disjoint groups of indices $(U_n)_{n\ge 1}$  where the mutual distance within group is smaller than $e^{-r}$
and the distance between groups is larger than $e^{-r}$
\begin{equation}\begin{split}
\min_{n\ne q}\min_{k\in U_n, l\in U_q}|x_k-x_l| &\ge e^{-r},\\
\forall n, \quad  \max_{k,l\in U_n}|x_k-x_l| &< e^{-r}.
\end{split}
\end{equation}
\begin{figure}[h] 
\centering
\subfloat[Configuration with existing partition]{\begin{tikzpicture}[xscale=0.75,yscale=0.75] 
\tikzstyle{sommet}=[circle,draw,fill=yellow,scale=0.4] 
\draw[color=red,fill=blue!30,rounded corners=1mm] (0,0) -- (10,0) -- (10,5) -- (0,5) -- cycle;
\node[sommet] (A) at (1,1.5){ $\mathbf{+}$};
\node[sommet] (B) at (8.5,2){ $\mathbf{+}$};
\node[sommet] (C) at (8,1.6){ $\mathbf{+}$};
\node[sommet] (D) at (4.5,3.9){ $\mathbf{+}$};
\node[sommet] (F) at (1.5,2.3){ $\mathbf{-}$};
\node[sommet] (G) at (8.2,2.5){ $\mathbf{-}$};
\node[sommet] (H) at (7.5,2){ $\mathbf{-}$};
\node[sommet] (I) at (2,1.8){ $\mathbf{-}$};
\draw[<->] (9.1,1) -- (9.1,2) node[right]{$e^{-t}$}-- (9.1,3);
\draw[<->] (5.5,4) -- (6.45,3.5) node[above]{$\,\,\, e^{-t}$}-- (7.5,3);
\draw (1.4,1.8) circle (1);
\draw (2,1) node[below]{$U_1$};
\draw (4.5,3.9) circle (0.9);
\draw (4.5,2.9) node[below]{$U_2$};
\draw (8,2) circle (1);
\draw (8,1) node[below]{$U_3$};
\end{tikzpicture}
}
\,\hspace{1cm}\subfloat[Configuration with no possible partition]{\begin{tikzpicture}[xscale=0.75,yscale=0.75] 
\tikzstyle{sommet}=[circle,draw,fill=yellow,scale=0.4] 
\draw[color=red,fill=blue!30,rounded corners=1mm] (0,0) -- (10,0) -- (10,5) -- (0,5) -- cycle;
\node[sommet] (A) at (1,1.5){ $\mathbf{+}$};
\node[sommet] (B) at (4,2){ $\mathbf{+}$};
\node[sommet] (C) at (7.2,1.6){ $\mathbf{+}$};
\node[sommet] (D) at (2.8,2.7){ $\mathbf{+}$};
\node[sommet] (F) at (1.5,2.3){ $\mathbf{-}$};
\node[sommet] (G) at (8.2,2.5){ $\mathbf{-}$};
\node[sommet] (H) at (5,2.3){ $\mathbf{-}$};
\node[sommet] (I) at (6,2.5){ $\mathbf{-}$};
\draw[<->] (4,3) -- (5,3) node[above]{$e^{-t}$}-- (6,3);
\end{tikzpicture}}
\caption{Partition into clusters}
\label{fig:partition}
\end{figure}
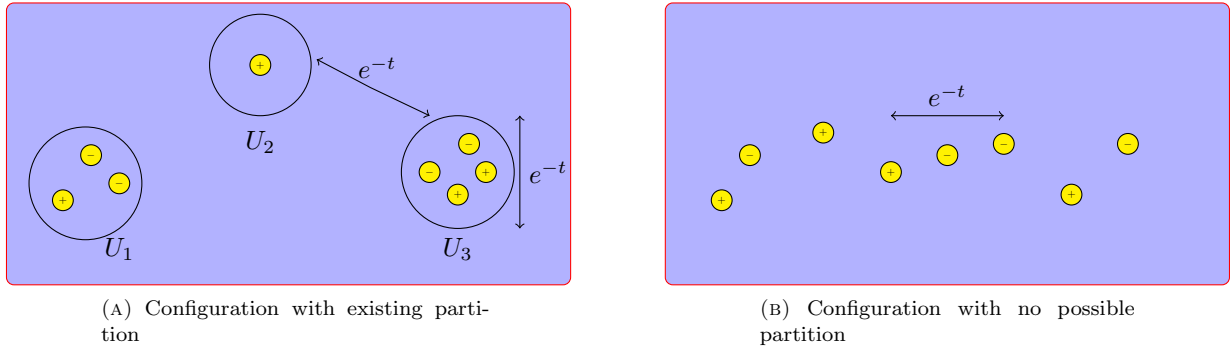
Such a subdivision may not exist but is unique when it does, see Figure \ref{fig:partition}.
A sufficient condition for existence is that $e^{-r}$ is not an approximation of any of the distances between points 
\begin{equation}
\forall k,l \in \lint 1, i\rint, \quad e^{r}|x_k-x_l|\in \R_+ \setminus \left[ \frac{1}{i-1}, 1\right).
\end{equation}
Thus in particular the set of $r$ for which the partition is not defined has Lebesgue measure bounded above by  $\frac{i(i-1)}{2}\log (i-1)$. We let $\cA$ be the set of $r\in \R$ for which the partition exists. 
We write 
\begin{equation}
Y^{(r)}:=\sum_{n\ge 1} Y^{(r)}_n \quad \text{ where } \quad Y^{(r)}_n:=\sum_{k\in U_n} \gep_k Z^{(r)}(x_k).
\end{equation}
As $Q$ is supported on $B(0,1)$  we have  $\Cov(Y^{(r)}_{n_1}, Y^{(r)}_{n_2})=0$ for $n_1\ne n_2$ and
using the fact that $Q$ is Lipschitz on the unit ball, 
\begin{align*}
\Var(Y^{(r)}_n)= &\sum_{k,l\in U_n} \gep_k\gep_l Q(e^r|x_k-x_l|) \\
=& \Big(\sum_{k\in U_n} \gep_k\Big)^2+   \sum_{k,l\in U_n} \gep_k\gep_l \Big(Q(e^r|x_k-x_l|)-1 \Big) \\
\ge& \Big(\sum_{k\in U_n} \gep_k\Big)^2 - C\sum_{k,l\in U_n} e^r|x_k-x_l|.  
\end{align*}
The term  $\sum_{k,l\in U_n} e^r|x_k-x_l| $  can be neglected. More precisely we have  
\begin{equation}
 \int_{\cA} \Big(\sum_n\sum_{k,l\in U_n} e^r|x_k-x_l|\Big)\dd r 
 \le \int_{\bbR_+}\sum_{1\le k,l\le i} e^r|x_k-x_l|\ind_{\{|x_k-x_l|< e^{-r}\}} \dd r \le C  i(i-1) .
\end{equation}
Thus for any value of $u$ and $s$ we have (with a constant that depends on $i$)
\begin{equation}
   \int^u_s \Var (Y^{(r)})\,\dd r \ge \int_{[s,u]\cap \cA} \sum_{n\ge 1} \left(\sum_{k\in U_n} \gep_k\right)^2 \,\dd r-C.
\end{equation}
Now if $\sum \gep_k \ne 0$ we have necessarily a group with non-cancelling signs and thus 
 $\sum_{n\ge 1} \left(\sum_{k\in U_n}\gep_k \right)^2\ge 1,$
 for all $r\in \cA$. This is sufficient to prove the first line of  \eqref{aprouve} 
using the fact that the measure of $[u,s]\setminus \cA$ is bounded by a constant.
 If $\sum \gep_k = 0$, $r\in \cA$ and  $e^{-r}< 2m({\bf x})/i$ then necessarily there must be a group with non zero sign ($\sum_{k\in U_n} \gep_k\ne 0$). Indeed, if not, each group $U_n$ contains the same number of $x_k$'s with positive $\gep_k$ or negative  $\epsilon_k$. We can then pair these $x_k$ inside each $U_n$ so as to obtain $\sigma\in\mathcal{S}_p$ (where $i=2p$) such that
 $$\sum_{k=1}^p|x_k-x_{p+\sigma(k)}|\leq p e^{-r}.$$
 Thus, if all groups have zero sign then  $m({\bf x})\leq i e^{-r}/2$. On the complement of this event,   $\sum_{n\ge 1} \left(\sum_{k\in U_n}\gep_k \right)^2\ge 1,$ from which we deduce  the second line of 
 \eqref{aprouve}.
\qed

\subsection{Proof of Proposition \ref{prop:cumul}}\label{sub:proof}

\subsubsection{Presentation of the induction setup}

We  prove the result by induction on $i$ and omit all dependences in $\gb$ in the notation. For $i=2$ and from \eqref{compute}, the functions $F^{(2,0)},F^{(2,1)}$ are given by
$$F^{(i,1)}(x_1,x_2,u,t)=-F^{(i,0)}(x_1,x_2,u,t)= \frac{\gb^2}{4} Q_u(x_1,x_2).$$
In that special case there is no dependence in $t$,  so that  \eqref{limit}   is trivial.  The estimate \eqref{didness} follows from our assumption on $Q_u$, namely Assumption \ref{ass:WN} Item 1.  The inequality  \eqref{propz} is also easy to check.
For the induction step,  we are going to prove that for every $i\ge 3$, 
$j\le \lfloor i/2 \rfloor$, setting 
$Y_p(\mathbf{x},u):=\sum_{k=1}^{i} \gl^{(p)}_k X_{u }(x_k) $
we have
\begin{equation}\label{daform}
\langle M^{(i-j,t)},M^{(j,t)}\rangle_t=\sum_{p=0}^{\lfloor i/2 \rfloor} \int^t_0 \int_{I^i} F^{(i,j,p)}(\bx,u,t)e^{\sum_{k=1}^i\frac{\gb^2}{2}  K_u(x_k,x_k)}  \cos( \gb Y_p(\mathbf{x},u) ) \mu(\dd \bx) \, \dd u,
\end{equation}
for a function $F^{(i,j,p)}$ satisfying  \eqref{didness}-\eqref{propz}. The statement for $A^{(i)}_t$ immediately follows.
The first task is to obtain an expression for $F^{(i,j,p)}(\bx,u,t)$ in terms of the function obtained in previous iterations. We set $V_p(\mathbf{x},u):=\E\Big[Y_p(\mathbf{x},u)^2\Big]$. As a consequence of our assumption \eqref{formz}, we can write 
$M^{(i-j,t)}$ as
 \begin{multline}\label{defmk}
M^{(i-j,t)}_s:=\sum_{p=0}^{\lfloor (i-j)/2 \rfloor} \int^t_0 \int_{I^{i-j}} F^{(i-j,p)}(\bx,u,t) e^{ \frac{\beta^2}{2}\left(\sum_{k=1}^{i-j}K_u(x_k,x_k)-V_p(\mathbf{x},u)\right)}\\
\times \Big(
  \cos\left( \gb Y_p(\mathbf{x},u\wedge s) \right)e^{\frac{\beta^2}{2}V_p(\mathbf{x},u\wedge s) }-  1     \Big) 
 \mu(\dd \bx ) \dd u.
\end{multline}
The martingale $M^{(j,t)}_s$ can be written as a similar sum over $q\in \lint  0,\lfloor j/2 \rfloor \rint$ and thus 
$\langle M^{(i-j,t)},M^{(j,t)}\rangle_t$ is the sum over $p$ and $q$  of the quadratic variations generated by the corresponding terms in the sum. We are going to prove that each of these quadratic variations can be written as \eqref{daform} with  functions that satisfy the 
properties specified in Proposition \ref{prop:cumul}.

\medskip

In what follows, for the sake of readability, we make a change of variables and replace the pair $(i-j,j)$ with $(i,j)$ (hence  $i+j$ now has the role formerly played by $i$). From now on, we also use the notation  $\mathbf{x}_1=(x_1,\dots,x_i)$, $\mathbf{x}_2=(x_{i+1},\dots,x_{i+j})$ and $\bx=(\bx_1,\bx_2)$. 
 We let  
$U_{i,p}(\bx_1,u,v)$  and $U_{j,q}(\bx_2,u,v)$ be defined as in \eqref{minusV}, more precisely we have  
\begin{equation}\label{defU}
U_{j,q}(\bx_2,u,v):=2 \sum\limits_{1\le k<l \le j}\gl^{(q)}_k\gl^{(q)}_l \int_{u}^{v} Q_r(x_{k+i},x_{i+l})\dd r,  \quad \text{ and } \quad U_{j,q}(\bx_{2},v):=U_{j,q}(\bx_{2},0,v).
\end{equation}
Expanding all products, our task is reduced to the study of the bracket of martingales of the following type
\begin{equation}\begin{split}\label{saltype}
 W_s&= \int^t_0 \int_{I^i}   F(\mathbf{x}_1,u,t) e^{ - \frac{\gb^2}{2} U_{i,p}(\bx_1,u)}  \Big(
  \cos\left( \gb Y_p(\mathbf{x}_1,u\wedge s) \right)e^{\frac{\beta^2}{2}V_p(\mathbf{x}_1,u\wedge s)}-  1    \Big)   
 \mu (\dd \bx_1) \dd u,  \\
R_s&= \int^t_0 \int_{I^j}   G(\mathbf{x}_2,u,t) e^{ - \frac{\gb^2}{2} U_{j,q}(\bx_2,u)} \Big(
  \cos\left( \gb Y_q(\mathbf{x}_2,u\wedge s) \right)e^{\frac{\beta^2}{2}V_q(\mathbf{x}_2,u\wedge s)}-   1   \Big)  
 \mu (\dd \mathbf{x}_2)  \dd u  ,
\end{split}
\end{equation}
where $i,j\ge 2$ and functions $F$ and $G$ both satisfy \eqref{didness}-\eqref{propz} for $i$ and $j$ respectively. Note that, to cover  the case $j=1$, we   also need
to study the bracket $\langle W, M^{(1)} \rangle_t$, but this is comparatively simpler (the case is treated starting from Equation \eqref{isier}).
The bracket can be computed  using \eqref{correlsnh}  
\begin{multline}\label{tam}
\dd \langle W,R\rangle_s= \int_{I^{i+j}} \gb^2 \left(\sum_{k=1}^i \sum_{l=i+1}^{i+j} \gl^{(p)}_k \gl^{(q)}_l Q_s(x_k,x_l)\right) \left(\int^t_s  F(\mathbf{x}_1,u,t) e^{ - \frac{\gb^2}{2} U_{i,p}(\bx_1,s,u)} \dd u\right)\\
\left(\int^t_s G(\mathbf{x}_2,v,t)  e^{ - \frac{\gb^2}{2} U_{j,q}(\bx_1,s,v)}\dd v \right) e^{\frac{\gb^2}{2}\sum_{k=1}^{i+j}K_s(x_k,x_k)} \\
\times
 \sin\Big( \gb Y_p(\mathbf{x}_1, s) \Big)
\sin\Big( \gb Y_q(\mathbf{x}_2, s) \Big)\mu (\dd \mathbf{x}_1)\mu(\dd \mathbf{x}_2)\dd s.
\end{multline}
Now we use trigonometry to change the product of sines  into a sum
which, after permutation of the coordinates, can be rewritten as the sum of two terms of the form
$\cos (\gb Y_{r}(\mathbf{x}, s))$, one term with $r=p+q$ and the other with $r=\min(p+j-q, q+i-p)$.
It remains to show  that the function obtained in front of those terms, namely 
 (we omit sign and a factor $\gb^2/2$ for obvious reasons)
\begin{multline}\label{defH}
 H(\mathbf{x}_1,\mathbf{x}_2,s,t):=  \left(\sum_{k=1}^i \sum_{l=i+1}^{i+j} \gl^{(p)}_k \gl^{(q)}_l Q_s(x_k,x_l)\right)  \left(\int^t_s  F(\mathbf{x}_1,u,t) e^{ - \frac{\gb^2}{2} U_{i,p}(\bx_1,s,u)} \dd u\right)\\
 \times \left(\int^t_s G(\mathbf{x}_2,v,t)  e^{-\frac{\gb^2}{2} U_{i,q}(\bx_2,s,v)}\dd v \right),
\end{multline}
satisfies \eqref{didness},\eqref{limit} and \eqref{propz}.

\subsubsection{ Proof of \eqref{didness} and \eqref{limit}}
This is the easier part.
It is sufficient to prove that 
\begin{equation}\label{leetruc}
\int^t_s  |F(\mathbf{x}_1,u,t) |e^{ - \frac{\gb^2}{2} U_{i,p}(\bx_1,s,u)} \dd u \le C(\bx_1)e^{-s} 
\end{equation}
and  converges (because this also implies convergence without $|\cdot|$ ) when $t$ tends to infinity (and the same for the term containing $G$).
This is obtained by observing that for $\bx_1$ fixed, $U_{i,p}(\bx_1,s,u)$ is uniformly bounded in $u$ and $s$ (e.g. by
$2 \sum_{1\le k<l \le i}\int_{0}^{\infty} |Q_r(x_k,x_l)|\dd r$). 
Hence using Assumption \eqref{didness} for $F$ and dominated convergence, we prove both convergence and \eqref{leetruc}.
 \qed
 
 \subsubsection{Proof of \eqref{propz}}

 For this statement we are going to rely on Lemma \ref{simplifbis}.
We replace $F(\mathbf{x}_1,u,t)$ and $G(\mathbf{x}_2,v,t)$ respectively by $F(\mathbf{x}_1,u):=\sup_{t:t\geq u}|F(\mathbf{x}_1,u,t)|$ and $G(\mathbf{x}_2,v):=\sup_{t:t\geq v}|G(\mathbf{x}_2,v,t)|$ and replace $t$ by $\infty$ in the integral bounds in \eqref{defH}. Of course, if we can prove appropriate bounds for the resulting quantity, this is enough for our claim.
We start with \eqref{propz},
and the easier case, when $p\ne i/2$ nor $q\ne j/2$. In that case, using \eqref{goodzbis},  we can conclude if we  show that there exists a constant such that for any $k\in \lint 1,i\rint$ and $l\in \lint i+1,i+j \rint$ and $N\in \lint 1,i+j\rint$ (not necessarily distinct of $k$ or $l$)  and any $x_N\in I$ 
\begin{equation}\label{defH0}
 \int_{I^{i+j-1}} \!\!\!\!\!\!\!\!\! |Q_s(x_k,x_l)| \! \left[\int^\infty_s\ \!\!\!\!  F(\mathbf{x}_1,u) e^{\frac{\gb^2}{2}(i-1)(u-s)} \dd u  \right]\!\!
 \left[\int^\infty_s \!\!\!\! G(\mathbf{x}_2,v)  e^{\frac{\gb^2}{2}(j-1)(v-s)}\dd v  \right]\!\! \mu(\dd \bx^{(N)})\le C e^{-(i+j-1)s}.
\end{equation}
Now assume without loss of generality that $N\ge (i+1)$ (the two integrals have symmetric roles).
Then we have from our induction hypothesis (for the first and last inequality) and Assumption \eqref{ass:WN} (for the second inequality) for every $x_k$, $x_l$ and $x_N$ respectively 
\begin{equation}
\begin{split}\label{bigthree}
\int_{I^{i-1}} F(\mathbf{x}_1,u) \mu(\dd \bx^{(k)}_1)&\le C_1 e^{-(i-1)u},\\
   \int_{I} |Q_s(x_k,x_l)| \mu(\dd x_k) &\le C_2e^{-s},\\
  \int_{I^{j-1}} G(\mathbf{x}_2,v) \mu(\dd\bx^{(N)}_2) &\le C_3 e^{-(j-1)v}.
  \end{split}
 \end{equation}
Combining the three inequalities, we get \eqref{defH0} by using  Fubini and integrating with respect to $\bx^{(k)}_1,x_k,\bx^{(N)}_2$ (in that order). Therefore, we obtain that the l.h.s. of \eqref{defH} is smaller than
\begin{equation}
 C_4 e^{-s} \int^\infty_s \int^\infty_s  e^{-(i-1)u} e^{-(j-1)v}\dd u \dd v\le C_5 e^{-(i+j-1)s}.
\end{equation}
 We can now move to the proof of \eqref{propz} when $i=2p$ or $j=2q$  (or both).
Let us start with the case where  $i=2p$, $j\ne 2q$ first.
We have 
\begin{equation}\label{resum}
 \left(\sum_{k=1}^i \sum_{l=i+1}^{i+j} \gl^{(p)}_k \gl^{(q)}_l Q_s(x_k,x_l)\right) 
 = -\sum_{l=i+1}^{i+j}\sum_{k=1}^p \gl^{(q)}_l \left( Q_s(x_{k},x_l)-Q_s(x_{\sigma(k)+p},x_l)\right),
\end{equation}
where $\sigma$ can be chosen as the permutation for which the minimum in \eqref{wasse} is attained. By symmetry we can assume that $\sigma$ is the identity, meaning that all our integrals will be restricted to the set $A_{\mathrm{Id}}:=\{ \bx_1 \ : \ \sigma=\mathrm{Id} \}$. This implies in particular that 
\begin{equation}
\forall k\in \lint 1, p \rint, \quad |x_{k}-x_{k+p}|\le m(\mathbf{x}_1).
\end{equation}
Using the triangle inequality and Lemma \ref{simplifbis}
we can conclude provided we prove a uniform bound 
\begin{multline}\label{defH2}
 \int_{I^{i+j-1}} \frac{|Q_s(x_k,x_l)- Q_s(x_{k+p},x_l)|}{ \min\left(m(\mathbf{x}_1)e^s,1\right)^{\frac{\gb^2}{2}}} \ind_{\{x_1\in A_{\mathrm{Id}}\}}\\
 \times\left[\int^\infty_s   F(\mathbf{x}_1,u) e^{\frac{\gb^2}{2}(i-1)(u-s)} \dd u\right]
 \left[\int^\infty_s G(\mathbf{x}_2,v)  e^{\frac{\gb^2}{2}(i-1)(v-s)}\dd v \right] \mu(\dd \bx^{(N)})\le C e^{-(i+j-1)s},
\end{multline}
valid uniformly in the choice of $k$, $l$ and $N$.
Now we claim that 
\begin{equation}\label{bornladeriv}
 |Q_s(x_k,x_l)- Q_s(x_{k+p},x_l)|\le \begin{cases}   C[(1+ e^s|x_k-x_l|)^{-2}+ (1+ e^s|x_{k+p}-x_l|)^{-2}] &\text{if } m(\bx_1)\ge  e^{-s}, \\
                                        C   m(\bx_1)e^{s}(1+ e^s|x_k-x_l|)^{-2}   \quad &\text{if }    m(\bx_1)\le e^{-s}.
                            \end{cases}
\end{equation}
Both inequalities comes from our Assumption \eqref{laderive}, the first line coming from the control on $Q$ and the second from the control on its derivative 
\begin{equation*} 
|Q_s(x_k,x_l)- Q_s(x_{k+p},x_l)|\le m(\bx_1)\max_{a\in[x_k,x_{k+p}]} |\partial_x  Q_s(a,x_l)|\le C m(\bx_1)e^s (1+ e^s |x_k-x_l|)^{-2} 
\end{equation*}
and noticing that if $m(\bx_1)\le e^{-s}$ then  $\max_{a\in[x_k,x_{k+p}]}  e^s (1+ e^s |a-x_l|)^{-2}  $   is of the same order as $e^s (1+ e^s |x_k-x_l|)^{-2}$. Because $\gb^2\le 2$ and for $m(\bx_1)\le e^{-s}$, the contribution of the term $ m(\bx_1)  e^{s}$ coming from \eqref{bornladeriv} cancels out the one appearing in  
\eqref{defH2}. 

\medskip

This is the part of the computation where the assumption  $d=1$ is crucially needed. Of course for $d\ge 2$ this cancelation also occurs for $\beta^2\le 2$, but the regime of interest is in that case $\beta^2\in [d,2d)$.
As a consequence of the above considerations,  \eqref{defH2} is proved if we can show that
\begin{multline}\label{defH2prim}
 \int_{I^{i+j-1}} (1+e^s |x_k-x_l|)^{-2}  \left[\int^\infty_s   F(\mathbf{x}_1,u) e^{\frac{\gb^2}{2}(i-1)(u-s)} \dd u\right]\\
 \times
 \left[\int^\infty_s G(\mathbf{x}_2,v)  e^{\frac{\gb^2}{2}(i-1)(v-s)}\dd v \right] \mu(\dd \bx^{(N)})\le C e^{-(i+j-1)s},
\end{multline}
which can be done exactly like \eqref{defH0} by observing that $\int_I(1+e^s |x_k-x_l|)^{-2}\mu(\dd x_k)\le C e^{-s}.$

\medskip

\noindent Finally let us consider the case $i=2p$, $j=2q$. 
Again by symmetry, we can assume that both permutations involved in the definition of $m(\bx_1)$ and $m(\bx_2)$ are the identity, which ensures that
\begin{equation}\label{smalldist}
\forall k,l, \ |x_{k}-x_{k+p}|\le m(\mathbf{x}_1), \quad |x_{l}-x_{l+q}|\le m(\mathbf{x}_2) .
\end{equation}
In that case similarly to \eqref{resum} 
\begin{equation}\label{resum2}
 \left(\sum_{k=1}^i \sum_{l=i+1}^{i+j} \gl^{(p)}_k \gl^{(q)}_l Q_s(x_k,x_l)\right) 
 = \sum_{l=i+1}^{i+q}\sum_{k=1}^p \left( Q_s(x_{k},x_l)-Q_s(x_{k+p},x_l)-
 Q_s(x_{k},x_{l+q})+ Q_s(x_{k+p},x_{l+q})\right).
\end{equation}
Hence  the inequality we have to prove becomes (with $B_{\mathrm{Id}}:=\{ \bx_2 \ : \ \sigma=\mathrm{Id} \}$)
\begin{multline}\label{XXX2}
\int_{I^{i+j-1}}\frac{| Q_s(x_{k},x_l)-Q_s(x_{k+p},x_l)-
 Q_s(x_{k},x_{l+q})+ Q_s(x_{k+p},x_{l+q})|}{  \left(\min\left(m(\mathbf{x}_1)e^s,1\right) \min\left(m(\mathbf{x}_2)e^s,1\right)\right)^{\gb^2/2}} \ind_{\{\bx_1\in A_{\mathrm{Id}}\ ; \  \bx_2 \in B_{\mathrm {Id}} \}}   \\ \times
  \left(\int^\infty_s  F(\mathbf{x}_1,u)   e^{-\frac{\gb^2(i-1)(u-s)}{2}}  \dd u\right)
 \left(\int^\infty_s G(\mathbf{x}_2,v)   e^{-\frac{\gb^2(j-1)(v-s)}{2}}\dd v \right)\mu(\dd \bx^{(N)}) \le Ce^{-(i+j-1)s}
\end{multline}
and has to be valid for every choice of $k$, $l$, and $N$.
 Here again we can conclude provided that we can replace the fraction by  $C(1+e^s |x_k-x_l|)^{-2}$.
 Hence we need to
 show that the numerator in the fraction is smaller than $ C m(\mathbf{x}_1)  m(\mathbf{x}_2) (1+e^s |x_k-x_l|)^{-2} $ when both 
 $m(\mathbf{x}_1)\le e^{-s}$ and $m(\mathbf{x}_2)\le e^{-s}$ 
 (other cases are similar but simpler).
 In this case, elementary differential calculus and our bound on the second derivative 
 \eqref{laderive} show  that the numerator in the fraction above is smaller than   
\begin{multline}
| Q_s(x_{k},x_l)-Q_s(x_{k+p},x_l)-
 Q_s(x_{k},x_{l+q})+ Q_s(x_{k+p},x_{l+q})|\\
\le C m(\mathbf{x}_1)  m(\mathbf{x}_2) e^{2s}\maxtwo{a\in[x_{k}, x_{k+p}]}{b\in [x_{l}, x_{l+q}]} C(1+e^{2s}|b-a|)^{-2}\le
 C m(\mathbf{x}_1)  m(\mathbf{x}_2) (1+e^{s}|x_k-x_l|)^{-2},
\end{multline}
where the last inequality uses our assumption on $m(\mathbf{x}_1)$ and  $m(\mathbf{x}_2)$ as well as \eqref{resum2}. This concludes our proof of \eqref{propz} except from the  case $j=1$  left aside.

In the case $j=1$ we need to compute a bracket of the type $\langle W, M^{(1)}\rangle_t$ 
where $W$ is as above and $M^{(1)}_t$ is our original martingale (minus its mean).
We have 
\begin{multline}\label{isier}
\dd \langle W,M^{(1)}\rangle_s= \int_{I^{i+1}} \gb^2 \left(\sum_{k=1}^i   \gl^{(p)}_k Q_s(x_k,x_{i+1})\right) \left(\int^t_s  F(\mathbf{x}_1, u,t) e^{-\frac{\gb^2}{2} U_{i,p}(\bx_1,s,u)} \dd u\right)\\
 e^{\frac{\gb^2 }{2}\sum_{k=1}^{i+1}K_s(x_k,x_k)} 
 \sin\left( \gb Y_p(\bx_1,s)\right)
\sin\left( \gb X_{s}(x_{i+1})\right) \mu(\dd  \mathbf{x}_1)\mu(\dd x_{i+1})\dd s.
\end{multline}
Using the notation $\bx=(\bx_1,x_{i+1})$,
the function for which convergence and domination have to be proved is 
\begin{equation}\label{defHdeuz}
 H(\bx,s,t):= \left(\sum_{k=1}^i  \gl^{(p)}_k  Q_s(x_k,x_{i+1})\right)\left( \int^t_s  F(\mathbf{x}_1,u,t) e^{- \frac{\gb^2}{2} U_{i,p}(\mathbf{x}_1,s,u)} \dd u\right).
\end{equation}
The proof of \eqref{didness} and \eqref{limit} for this $H$ works as previously.
Now concerning \eqref{propz}, we prove the bound for 
\begin{equation}\label{defHHH}
 H(\bx,s)= \left|\sum_{k=1}^i  \gl^{(p)}_k  Q_s(x_k,x_{i+1})\right|\left( \int^{\infty}_s  F(\mathbf{x}_1,u) e^{- \frac{\gb^2}{2} U_{i,p}(\mathbf{x}_1,s,u)} \dd u\right).
\end{equation}
with $F(\mathbf{x}_1,u):=\sup_{t\ge u} F(\mathbf{x}_1,u,t)$. The proof is the same as for the $j\ge 2$ case.
\qed 
\subsubsection{Proof of Item (3)} For this, the starting point is again to observe that the validity of \eqref{formz} entails that each martingale $M^{(i,t)}$ takes the form \eqref{defmk}, and thus 
 the candidate for the limit  is the martingale 

 \begin{equation}\label{defmk2}
\bar M^{(i)}_s:=\sum_{p=0}^{\lfloor (i)/2 \rfloor} \int^\infty_0 \int_{I^{i}} \bar F^{(i,p)}(\bx,u) e^{ -\frac{\beta^2}{2}U_{i,p}(\mathbf{x},u)}\Big(
  \cos\left( \gb Y_p(\mathbf{x},u\wedge s) \right)e^{\frac{\beta^2}{2}V_p(\mathbf{x},u\wedge s) }-  1     \Big) 
 \mu(\dd \bx ) \dd u.
\end{equation}
 We are going to prove the convergence statement for the bracket, the other one being only simpler. 
The  bracket derivative $\partial_s \langle \bar M^{(i)}-\bar M^{(i,t)} \rangle$ is given by
a sum of terms of the form (recall \eqref{tam}, we adopt the convention $F^{(i,p)}(\bx,u,t)=0$ for $u>t$)
\begin{multline}\label{tam2}
\int_{I^{2i}} \left(\sum_{k=1}^i \sum_{l=i+1}^{2i} \gl^{(p)}_k \gl^{(q)}_l Q_s(x_k,x_l)\right) \left(\int^\infty_s  (\bar  F^{(i,p)}(\bx_1,u)-   F^{(i,p)}(\bx_1,u,t)) e^{ - \frac{\gb^2}{2} U_{i,p}(\bx_1,s,u)} \dd u\right)\\
\left(\int^\infty_s (\bar  F^{(i,q)}(\bx_2,u)-   F^{(i,q)}(\bx_2,u,t)) e^{ - \frac{\gb^2}{2} U_{i,q}(\bx_2,s,u)} \dd u\right) e^{\frac{\gb^2}{2}\sum_{k=1}^{i+j}K_s(x_k,x_k)} \\
\times
 \sin\Big( \gb Y_p(\mathbf{x}_1, s) \Big)
\sin\Big( \gb Y_q(\mathbf{x}_2, s) \Big)\mu (\dd \mathbf{x}_1)\mu(\dd \mathbf{x}_2).
\end{multline}
We show that each term converges to zero and remains uniformly bounded for  $s\in[0,K]$. By  dominated convergence it is sufficient to show that (recall $F^{(i,p)}(\bx,u):= \sup_{t>u}|F^{(i,p)}(\bx,u,t)|$)
\begin{multline}
\int_{I^{2i}} \left(\sum_{k=1}^i \sum_{l=i+1}^{2i} \gl^{(p)}_k \gl^{(q)}_l Q_s(x_k,x_l)\right) \left(\int^\infty_s    F^{(i,p)}(\bx,u) e^{ - \frac{\gb^2}{2} U_{i,p}(\bx_1,s,u)} \dd u\right)\\
\left(\int^\infty_s F^{(i,q)}(\bx,u) e^{ - \frac{\gb^2}{2} U_{i,p}(\bx_1,s,u)} \dd u\right)\mu (\dd \mathbf{x}_1)\mu(\dd \mathbf{x}_2)<\infty.
\end{multline}
This can be performed by repeating the computation of the induction step in the proof of \eqref{propz}.\qed

\section{Convergence of the partition function: Proof of Theorem \ref{th:main}}\label{sec:partition}
We fix $n$ and $\beta^2 <2-\frac{1}{n}$. Let us consider some bounded measurable functional $f$ which we further assume to be measurable with respect to $\mathcal{F}_{s_0}$ for some $s_0>0$. Using Lemma \ref{lem:cum} we have

\begin{align*}
\mathcal{Z}^\beta_t(f):=\E[f e^{\alpha M_t}]e^{-\sum_{i=1}^{2n-1}\frac{\alpha^i}{i!} \mathcal{C}_i(M_t)}=\E\left[f e^{N^{(2n-1,\alpha,t)}_t-\frac{1}{2}\langle N^{(2n-1,\alpha,t)}\rangle_t} e^{ Q^{(2n-1,\alpha)}_t}\right]
 \end{align*}
 where $N^{(2n-1,\alpha,t)}_t$ and   $ Q^{(2n-1,\alpha)}_t$ are defined by \eqref{defn} and \eqref{defq}. Let us define for $s\leq t$
\begin{equation}\label{qnalphabis}
 Q^{(2n-1,\alpha)}_{s,t}:= \frac{1}{2}\sum_{k=2n}^{4n-2} \alpha^{k} 
   \sum_{l=k-2n+1}^{2n-1}\langle M^{(l,t)},M^{(k-l,t)} \rangle_s 
 \end{equation}
where $ Q^{(2n-1,\alpha)}_{t,t}= Q^{(2n-1,\alpha)}_{t}$. We can assume that $f\geq 1$  by adding a large constant to it. 
Now, because of our assumption on $f$ of $\mathcal{F}_{s_0}$-measurability and the martingale property of $(N^{(2n-1,\alpha,t)}_s)_{s\leq t }$, we have for $s_0\leq s \leq t$
\begin{equation}\label{wopop}
 \ln \left|\frac{\E\left[f e^{N^{(2n-1,\alpha,t)}_t-\frac{1}{2}\langle N^{(2n-1,\alpha,t)}\rangle_t} e^{ Q^{(2n-1,\alpha)}_{t,t}}\right]}{\E\left[f e^{N^{(2n-1,\alpha,t)}_s-\frac{1}{2}\langle N^{(2n-1,\alpha,t)}\rangle_s} e^{ Q^{(2n-1,\alpha)}_{s,t}}\right]}\right|
  \le  \| Q^{(2n-1,\alpha)}_{t,t}- Q^{(2n-1,\alpha)}_{s,t}\|_{\infty}.
\end{equation}
Hence we can rewrite 
\begin{equation}\label{okokok}
\ln \E[f e^{\alpha M_t}]-\sum_{i=1}^{2n-1}\frac{\alpha^i}{i!} \mathcal{C}_i(M_t)= \ln \E\left[f e^{N^{(2n-1,\alpha,t)}_s-\frac{1}{2}\langle N^{(2n-1,\alpha,t)}\rangle_s} e^{ Q^{(2n-1,\alpha)}_{s,t}}\right]+ A_2(s,t).
\end{equation}
with $|A_2(s,t)|\le \| Q^{(2n-1,\alpha)}_{t,t}- Q^{(2n-1,\alpha)}_{s,t}\|_{\infty}$.
We are going to show that 
\begin{equation}\label{firstconv}
  \lim_{t\to\infty} A_1(s,t):=\lim_{t\to \infty}\ln  \E\left[F e^{N^{(2n-1,\alpha,t)}_s-\frac{1}{2}\langle N^{(2n-1,\alpha,t)}\rangle_s} e^{ Q^{(2n-1,\alpha)}_{s,t}}\right]=:A_1(s)
  \end{equation}
  exists for every $s$ and that  
  \begin{equation}\label{broom}
\lim_{s\to \infty}\sup_{t\ge s}  \| Q^{(2n-1,\alpha)}_{t,t}- Q^{(2n-1,\alpha)}_{s,t}\|_{\infty}=0.
\end{equation}
This is sufficient to conclude because repeating the argument in \eqref{wopop} we have
\begin{equation}\label{brooom}
 |A_1(s_1,t)-A_1(s_2,t)|\le \| Q^{(2n-1,\alpha)}_{s_1,t}- Q^{(2n-1,\alpha)}_{s_2,t}\|_{\infty}.  
\end{equation}
As a consequence of \eqref{brooom} the function $A_1(s)$ is Cauchy. From \eqref{okokok} and elementary calculus we deduce that
\begin{equation}
 \lim_{t\to \infty}\log (\E[f(X)e^{\alpha M_t}])-\sum_{i=1}^{2n-1}\frac{\alpha^i}{i!} \mathcal{C}_i(M_t)=
 \lim_{s\to \infty} A_1(s).
\end{equation} 
The convergence \eqref{firstconv} follows from the fact that for fixed $s$, the martingale $N^{(2n-1,\alpha,t)}_s$ can be written as a sum $N^{(2n-1,\alpha,t)}_s=\sum_{i=1}^{2n-1}\alpha^iM^{(i,t)}_s $ (see \eqref{defn}). Proposition \ref{prop:cumul} Item (3) then ensures then convergence as $t\to \infty$ in the essential supremum norm of
$N^{(2n-1,\alpha,t)}_s$,  $\langle N^{(2n-1,\alpha,t)}\rangle_s$ and also that of $Q^{(2n-1,\alpha)}_{s,t}$ which is a linear combination of martingale brackets.

Finally we prove \eqref{broom}. From  \eqref{qnalphabis}, it is enough to prove the claim for $\langle M^{(l,t)},M^{(k-l,t)} \rangle_s$ with $k\in \lint 2n,4n-2\rint$ and $l\in \lint k-2n+1,2n-1 \rint$. Similarly to \eqref{formz}, such a bracket takes the form of a sum
 \begin{equation}\label{formzbis}
\langle M^{(l,t)},M^{(k-l,t)} \rangle_s=\sum_{p=0}^{\lfloor k/2 \rfloor} \int^s_0 \int_{I^k} F^{(k,p)}(\bx,u,t)e^{\frac{\gb^2}{2}\sum_{l=1}^kK_u(x_l,x_l)}
 \cos\left[ \gb\left(\sum_{j=1}^{k} \lambda^{(p)}_j X_u(x_j)  \right) \right]
 \mu(\dd \bx) \dd u.
\end{equation}
Then Proposition \ref{prop:cumul} item 2 provides the bound
\begin{align*}
\|\langle M^{(l,t)},M^{(k-l,t)} \rangle_t-\langle M^{(l,t)},M^{(k-l,t)} \rangle_s   \|_\infty \leq C \int_s^\infty e^{\frac{\beta^2}{2}k u-(k-1)u}\,\dd u .
\end{align*}
This is enough for our claim because $k\geq 2n$ and $\beta^2 <2-\frac{1}{n}$.
\qed

\section{Convergence of the charge distribution: Proof of Theorem \ref{th:mass}}\label{sec:charge}

To implement this proof, let us move back to \eqref{scroot}.
We need to prove the convergence and continuity in $\theta$ of 
$\bbE[e^{\alpha M^{(\beta,\theta)}_t}]/\bbE[e^{\alpha M^{(\beta)}_t}]$ with $M^{(\beta,\theta)}_t$ defined in \eqref{mbetateta}.
 We are going to proceed in two steps. First we prove that  
\begin{equation}\label{kobe}
\lim_{t\to \infty}\frac{\bbE[e^{\alpha M^{(\beta,\theta)}_t}]}{\bbE[e^{\alpha M^{(\beta)}_t}]}e^{\sum_{i=1}^{n-1}\frac{\alpha^{2i}}{(2i)!} \big(\cC^{(\beta)}_{2i}(t)-\cC^{(\theta,\beta)}_{2i}(t)\big)} 
\end{equation}
exists and is continuous in $\theta$ and then that this is also the case of  
\begin{equation}\label{lebron}
\lim_{t \to \infty}\sum_{i=1}^{n-1}\frac{\alpha^{2i}}{(2i)!} \big(\cC^{(\beta)}_{2i}(t)-\cC^{(\theta,\beta)}_{2i}(t)\big).
\end{equation}
Let us start with \eqref{kobe}. Note that adding an angle $\theta$ in all the $\cos$ does not alter the computation made in Section
\ref{sec:proofmain} and 
\ref{sec:partition}, and everything that has been proved concerning the martingale  $M^{(\beta)}_t$ in the previous section remains true for
$M^{(\beta,\theta)}_t$. In the following section,   we keep this convention of adding $\theta$  as a superscript  in the notation of a quantity  previously defined for $M^{(\beta)}_t$ and now considered for $M^{(\beta,\theta)}_t$.

 In particular if $M^{(i,t,\theta)}$ is the martingale recursively defined by \eqref{dafai}  starting with   $M^{(\beta,\theta)}_t$ then we have like in \eqref{daform}
\begin{equation}\label{larmoire}
 \langle M^{(i-j,t,\theta)},M^{(j,t,\theta)}\rangle_t \! =\!\sum_{p=0}^{\lfloor i/2 \rfloor} \int^t_0 \!\! \int_{I^i} \!\! F^{(i,j,p)}(\bx,u,t)e^{\sum\limits_{k=1}^i\frac{\gb^2}{2}  K_u(x_k,x_k)}  \!\!\!\!\!\cos\Big( \gb Y_p(\mathbf{x},u) \! + \! \sum^p_{k=1}\gl^{(p)}_k\theta(x_k) \Big) \mu(\dd \bx) \, \dd u,
\end{equation}
where the functions $F^{(i,j,p)}(\bx,u,t)$ \emph{do not} depend on $\theta$ and the $\gl^{(p)}_k$'s defined in \eqref{recoil}.
In particular this implies that the limit
\begin{equation}
\cZ(\alpha,\beta,\theta):=\bbE[e^{\alpha M^{(\beta,\theta)}_t}]e^{-\sum_{i=1}^{n-1}\frac{\alpha^{2i}}{(2i)!} \cC^{(\theta,\beta)}_{2i}(t)}
\end{equation} 
exists for the same reason as in the case  $\theta=0$. Checking the continuity in $
\theta$ requires to be cautious.
As a consequence of \eqref{larmoire}, if $Q^{(2n-1,\alpha,\theta)}_{s,t}$ is defined as in \eqref{qnalphabis} then the proof of \eqref{broom} shows that the bound does not depend on $\theta$ and we have 
\begin{equation}
\lim_{s\to \infty}\sup_{\theta\in C(I)}\sup_{t\ge s}  \| Q^{(2n-1,\alpha,\theta)}_{t,t}- Q^{(2n-1,\alpha,\theta)}_{s,t}\|_{\infty}=0
\end{equation}
where the supremum is taken over the set $C(I)$ of continuous bounded functions on $I$.
Hence (recall \eqref{wopop}) to have continuity of $\cZ(\alpha,\beta,\theta)$ it is sufficient to prove that for any $s$ the limit
\begin{equation}
 \lim_{t\to \infty}\E\left[ e^{N^{(2n-1,\alpha,\theta, t)}_s-\frac{1}{2}\langle N^{(2n-1,\alpha,\theta,t)}\rangle_s+ Q^{(2n-1,\alpha,\theta)}_{s,t}}\right]
 \end{equation}
 is continuous in $\theta$. Using only the fact that the $\cos$ function is Lipschitz and following the computations in the  proof of Item (3) of Proposition \ref{prop:cumul} (starting with \eqref{defmk2}) we obtain that ($|\cdot|_{\infty}$ denotes the usual supremum norm)
 \begin{equation}\begin{split}
  \| N^{(2n-1,\alpha,\theta, t)}_s-N^{(2n-1,\alpha,\theta', t)}_s\|_{\infty}&\le C_s |\theta-\theta'|_{\infty},\\
  \| \langle N^{(2n-1,\alpha,\theta,t)}\rangle_s-\langle N^{(2n-1,\alpha,\theta',t)}\rangle_s\|_{\infty} &\le C_s |\theta-\theta'|^2_{\infty},\\
  \|Q^{(2n-1,\alpha,\theta)}_{s,t}- Q^{(2n-1,\alpha,\theta')}_{s,t} \|_{\infty} &\le C_s |\theta-\theta'|^2_{\infty},
  \end{split}
 \end{equation} 
 which yields in particular the desired continuity statement for the limit.
 Now let us look at \eqref{lebron}. In order to prove existence of the limit and continuity at the same time let us look at 
 $\cC^{(\theta,\beta)}_{2i}(t)- \cC^{(\theta',\beta)}_{2i}(t)$.
 We observe first that
 \begin{equation}
 \bbE\left[
 \cos\left[ \beta\left(Y_p(\bx)+\sum_{k=1}^{i} \lambda^{(p)}_k \theta(x_k)\right)\right]\right]=  \cos\Big(\beta \sum_{k=1}^{i} \lambda^{(p)}_k \theta(x_k) \Big)e^{-\frac{\gb^2}{2}{\Var(Y_p(\bf x))}}.
 \end{equation}
Using \eqref{larmoire} and \eqref{formz},  we obtain that for $A^{(i,\theta)}_t$ being defined as \eqref{dafai}
\begin{multline}\label{ladaptun}
 i! (\cC^{(\theta,\beta)}_{i}(t) - \cC^{(\theta',\beta)}_{i}(t))= \bbE[A^{(i,\theta)}_t-
 A^{(i,\theta')}_t]\\
 \!\!\!\!\!\!\!
  \!\!\!\!\!\!\!
   \!\!\!\!\!\!\!
 =\sum_{p=0}^{\lfloor i/2 \rfloor} \int^t_0 \int_{I^i} F^{(i,p)}(\bx,u,t)e^{\frac{\gb^2}{2}\sum_{k=1}^iK_u(x_k,x_k)} \\ \times \left[\cos\Big(\beta \sum_{k=1}^{i} \lambda^{(p)}_k \theta(x_k) \Big)-\cos\Big(\beta \sum_{k=1}^{i} \lambda^{(p)}_k \theta'(x_k) \Big)\right]e^{-\frac{\gb^2}{2}{\Var(Y_p(\bf x))}}
 \mu(\dd \bx) \dd u.
\end{multline}

As we have convergence of $F^{(i,p)}(\bx,u,t)$ towards a limit, 
we want to use dominated convergence  for each terms of the sum in $p$.
Let us observe that 
\begin{equation}
\left|\cos\Big(\beta \sum_{k=1}^{i} \lambda^{(p)}_k \theta(x_k) \Big)-\cos\Big(\beta \sum_{k=1}^{i} \lambda^{(p)}_k \theta'(x_k)\Big)\right| \le \beta i|\theta-\theta'|_\infty.
\end{equation}
Hence from the computation of the Section \ref{sec:proofmain} when $p\ne i/2$, we can deduce that we have convergence towards a function which is Lipschitz continuous for the $|\theta|_\infty$ norm.
For the case $p=i/2$, using 
$$\cos(b)-\cos(a)=2\sin\big(\frac{a+b}{2}\big)\sin\big(\frac{a-b}{2}\big)\le \frac{1}{2}|a+b|\times |a-b|$$ 
we have
\begin{equation}
 \left|\cos\Big(\beta \sum_{k=1}^{i} \lambda^{(p)}_k \theta(x_k) \Big)-\cos\Big(\beta \sum_{k=1}^{i} \lambda^{(p)}_k \theta'(x_k)\Big)\right|
 \le  \frac{1}{2} \beta^2 \left|\sum_{k=1}^{i} \lambda^{(p)}_k (\theta(x_k)-\theta'(x_k))\right|
 \left|\sum_{k=1}^{i} \lambda^{(p)}_k(\theta(x_k)+\theta'(x_k))\right|.
\end{equation}
Recalling the definition of $m(\bf x)$
\eqref{wasse} and that of the H\"older norm, we observe that, for the optimal permutation $\sigma$ and for any $1/2$-H\"older continuous function $\eta$, we have
\begin{equation}
 \left|\sum_{k=1}^{p}\eta(x_k)-\eta(x_{p+\sigma(k)})\right|
 \le p \sum_{k=1}^{p}|\eta(x_k)-\eta(x_{p+\sigma(k)})|\le p \|\eta\|_{1/2} \sqrt{m(\bf x)},
\end{equation}
which applied to $\eta= \theta \pm \theta'$ yields 
\begin{equation}
 \left|\cos\Big(\beta \sum_{k=1}^{i} \lambda^{(p)}_k \theta(x_k) \Big)-\cos\Big(\beta \sum_{k=1}^{i} \lambda^{(p)}_k \theta'(x_k)\Big)\right|\le \frac{\beta^2}{2}
 \|\theta+\theta'\|_{1/2} \|\theta-\theta'\|_{1/2} m(\bf x).
\end{equation}
To conclude using dominated convergence, we only need to check that  (recall $F^{(i,p)}(\bx,u):= \sup_{t>u}F^{(i,p)}(\bx,u,t)$)
\begin{equation}
 \int^\infty_0 \int_{I^i} F^{(i,p)}(\bx,u)e^{-\frac{\gb^2}{2}U_{i,p}(\bx,u)} m(\bx) 
 \mu(\dd \bx) \dd u<\infty.
\end{equation}
Combining \eqref{secondestbis} (for $s=0$) and  \eqref{propz} 
the quantity we have to estimate is smaller than 
\begin{multline}
  \int^\infty_0 \int_{I^i} F^{(i,p)}(\bx,u)e^{-\frac{\gb^2}{2}U_{i,p}(\bx,u)} m(\bx)^{\frac{\beta^2}{2}}\mu(\dd \bx)  \dd u \\
 \le C \int^\infty_0 e^{\frac{\beta^{2}}{2}(i-1)u}  \int_{I^i} F^{(i,p)}(\bx,u)  \mu(\dd \bx)  \dd u
  \le  C' \int^\infty_0 e^{ \left(\frac{\beta^2}{2}-1\right) u}\dd u<\infty.
\end{multline}
Finally, the convergence in the Schwartz space of tempered distributions is then a consequence of Levy-Schwartz Theorem \cite[Th. 2.3]{bierme}.
\qed

\section{Correlation functions :  Proof of  Theorem \ref{th:lecorrels},}\label{sec:correl}

 Let us recall that in this section,  $m$ the number of extra charges arbitrarily placed in the system. Then, using the Sine-Gordon representation for \eqref{zhatcorrel}, we have 
\begin{equation}\label{thegordon}
  \hat{\mathcal{Z}}_{\alpha/2,\beta,t}({\bf z},\boldsymbol{\eta})=e^{ -\beta^2\sum_{l<l'}\eta_l\eta_{l'}K_t(z_l,z_{l'})}\E\left[  e^{\alpha \hat M^{({\bf z},\eta)}_t}\right],
\end{equation}
where (considering the analytic continuation of $\cos$ on $\bbC$)
\begin{equation}\label{hatm}
 \hat M^{({\bf z},\eta)}_t:= \int_I   \cos\left(\beta X_t(x)+ \bi 
 \beta^2 \sum_{l=1}^m \eta_l K_t(z_l,x)\right) e^{\frac{\beta^2}{2} K_t(x,x)}\dd \mu(x).
\end{equation}
Switching from \eqref{zhatcorrel} to \eqref{hatm} (or   \eqref{espcorrel}) can be done by observing that both expressions are analytic in $\boldsymbol{\eta}$ and coincide for $\boldsymbol{\eta}\in (i\R)^m$ by applying the Girsanov transform.

Then omitting the dependence in $({\bf z}, \eta)$ in the martingale notation, the quantity defined in \eqref{hatm} is the terminal value of the martingale $(\hat M^{(t)}_{s})_{s\in[0,t]}$ defined by

\begin{equation}
\hat M^{(t)}_s:= \int_I e^{\frac{\gb^2}{2}K_s(x,x)}\cos \left(\gb X_s(x)+ \bi\sum_{k=1}^m\gb^2 \eta_k K_t(x,z_k) \right) \dd x.
\end{equation}
For notational simplicity we set for the rest of this section ($(z_k,\eta_k)$ being fixed)  
$$ \psi_t(x):= \gb^2\sum_{k=1}^m \eta_k K_t(x,z_k) 
,$$
and $ \psi=\lim_{t\to \infty}\psi_t$. 
The cumulants of $\hat M^{(t)}_s$ can be obtained simply by repeating 
the computation made for those of $M^{(\beta)}_t$.
If $\hat \cC_i(t)$ denotes the $i$-th cumulant of $\hat M^{(t)}_s$ , starting with \eqref{formz} and repeating the computation  \eqref{ladaptun}, we obtain
\begin{equation}\label{ladapt2}
 i! (\hat \cC_{i}(t)-\cC_{i}(t)) 
 \\=\sum_{p=0}^{\lfloor i/2 \rfloor} \int^t_0 \int_{I^i} F^{(i,p)}(\bx,u,t)e^{-\frac{\beta^2}{2}U_{i,p}(\bx,u)}
 \left[\cosh\left( \sum_{k=1}^i\gl^{(p)}_k \psi_t(x_k)\right)-1\right]
 \mu(\dd \bx) \dd u.
\end{equation}
Theorem \ref{th:lecorrels} is proved if we can show that the two following claims hold:
\begin{enumerate}
\item For every $i$, $i! (\hat \cC_{i}(t)-\cC_{i}(t))$ converges when $t$ tend to infinity. 
\item  There exists $n\ge 0$ such that
\begin{equation}\label{alprev}
   \lim_{t \to \infty} \bbE \left[ e^{\alpha \hat M^{(t)}_t}\right] e^{-\sum_{i=1}^n \frac{\alpha^i}{i! }  \hat \cC_{i}(t)}<\infty.
\end{equation}
\end{enumerate}
Since  the proof of \eqref{alprev} is very similar to what is done in Section \ref{sec:partition}, we only provide details for the first point which is the convergence of  cumulant difference. 
 As the term
$ \left[\cosh\left( \sum_{k=1}^i\gl^{(p)}_k \psi_t(x_k)\right)-1\right]$ is not bounded uniformly in $t$, we require an extension of Proposition \ref{prop:cumul}.
Let us introduce the quantity $\Upsilon$ on $I$ and its extension to $I^i$  as
$$ \Upsilon(x)= \max_{k\in \lint 0, n\rint}(-\log |x-z_k| ) \quad  \text{ and }  \quad \Upsilon(\bx):=\sum_{k=1}^i W(x_k).$$ 
Note that it is easy to check that 
\begin{equation}\label{laproox}
\cosh\left( \sum_{k=1}^i\gl^{(p)}_k \psi_t(x_k)\right)\le C e^{\beta^2 \|\eta\|_{\infty} \Upsilon(\bx)}.
\end{equation}
Hence we need an adapted version of Proposition \ref{prop:cumul}, which allows us for the integration of the bound above, in order to proceed along the same proof as in Section \ref{sec:proofmain}. 

\begin{proposition}\label{prop:cumulbis}

Given  for every $i\ge 2$, $p\le i/2$, then
whenever $2\alpha+\beta^2<2$ 
there exists a constant $C=C(i,{\bf z}, \alpha,\beta)>0$ such that for any fixed $k$  
 we have  
\begin{equation}\label{propz2}
\sup_{x_k\in I} \int_{I^{i-1}} e^{\alpha (\Upsilon(\bx )-\Upsilon(x_k))}\sup_{t\geq u} |F^{(i,p)}(\bx,u,t)| \mu(\dd \bx^{(k)}) \le C e^{(i-1)(\alpha-1)u}.
\end{equation}

\end{proposition}

We postpone the proof of Proposition \ref{prop:cumulbis} and wish to use to prove that 
whenever $\beta^2(2 \|\eta\|_{\infty}+1)<2$ we have for all $i$
(recall $\bar F^{(i,p)}(\bx,u)=\lim_{t\to \infty}F^{(i,p)}(\bx,u,t)$)
\begin{equation}\label{ladapttr}
 \lim_{t\to \infty }i! (\hat \cC_{i}(t)-\cC_{i}(t)) 
 \\=\sum_{p=0}^{\lfloor i/2 \rfloor} \int^\infty_0 \int_{I^i} \bar F^{(i,p)}(\bx,u)e^{-\frac{\beta^2}{2} U_{i,p}(\bx,u)}
 \left[\cosh\left( \sum_{k=1}^i\gl^{(p)}_k \psi(x)\right)-1\right]
 \mu(\dd \bx) \dd u<\infty.
\end{equation}
Both the convergence and finiteness of the integrals are proved using dominated convergence starting with \eqref{ladapt2}. 
We proceed separately for each $i$ and $p$ and start with the easier case $i\ne 2p$. Setting $\alpha=\beta^2 \|\eta\|_{\infty}$ and  as usual $F^{(i,p)}(\bx,u):= \sup_{t\geq u} |F^{(i,p)}(\bx,u,t)|$
and using \eqref{goodzbis} and \eqref{laproox}, we obtain that the integrand in \eqref{ladapt2} is bounded  uniformly in $t$ by 
$$C F^{(i,p)}(\bx,u)e^{\frac{\beta^2}{2}(i-1)u}e^{\alpha W(\bx)}.
$$
According to \eqref{propz2} the integral of the above is (up to a constant factor) smaller than
\begin{equation}
 \left( \int_I e^{\alpha W(x_1)} \mu (\dd x_1)\right)\left(  \int_{0}^{\infty} 
 e^{(i-1)[\alpha+(\gb^2/2)-1]u} \dd u\right) <\infty,
\end{equation}
where the finiteness of the above is implied by $(2\alpha+\beta^2)<2$.
When $i=2p$ we need to be substantially more accurate in our approximation.
Instead of \eqref{laproox} we observe (e.g. considering separately the case of 
$|\sum_{k=1}^i\gl^{(p)}_k \psi_t(x_k)|\le 1$ and $|\sum_{k=1}^i\gl^{(p)}_k \psi_t(x_k)|\ge 1$) that 
\begin{equation}\label{laproox2}
\left[\cosh\left( \sum_{k=1}^i\gl^{(p)}_k \psi_t(x)\right)-1\right]\le C e^{\alpha \Upsilon(\bx)}\min\left[1,  \left(\sum_{k=1}^i\gl^{(p)}_k \psi_t(x_k)\right)\right]^2.
\end{equation}
To obtain a bound which is uniform in $t$ we use the fact that there exists $C$ such that for every $t\ge 0$
\begin{equation}
\sum_{k=1}^i\gl^{(p)}_k \psi_t(x_k) \le C \frac{m(\bx)}{D(\bx)},
\end{equation}
where $m(\bx)$ is the Wasserstein distance defined in \eqref{wasse} and $D(\bx)=\mintwo{i\in \lint 1,k \rint}{j\in \lint 1, n \rint} |x_i-z_j|.$
We can also replace the exponent $2$ in \eqref{laproox2} by $\beta^2/2$ (which is smaller than $1$ in any case) and ignore the $\min$.
Using \eqref{secondestbis}, the bound we obtain for the integrand is then
$$  F^{(i,p)}(\bx,u)e^{\frac{\beta^2}{2}(i-1)u}
e^{\alpha \Upsilon(\bx )}  \left(\min( m(\bx)^{-1}, e^{u} ) \left(\frac{m(\bx)}{D(\bx)}\right)\right)^{\beta^2/2}.$$
Now observe that the last term (between parentheses) is smaller than
\begin{equation}
\frac{1}{D(\bx)^{\beta^2/2}}\le  \sum_{k=1}^i e^{\frac{\beta^2}{2}\Upsilon(x_k)}.
\end{equation}
To conclude we notice that for each $k\in\lint 1,i\rint$ we have from \eqref{propz2}
\begin{multline}
 \int^{\infty}_0\int_{I^i}  F^{(i,p)}(\bx,u)e^{\frac{\beta^2}{2}(i-1)u}
e^{\alpha \Upsilon(\bx )+\frac{\beta^2}{2}\Upsilon(x_k)} \mu (\dd \bx ) \dd u\\
\le  C\left(\int_{I}  e^{\left(\alpha+\frac{\beta^2}{2}\right) \Upsilon(x_k)} \mu( \dd x_k)\right)\left( \int^{\infty}_0 e^{(i-1)[\alpha+\beta^2/2-1]u} \dd u\right),
\end{multline}
where the finiteness of the first and second integrals are both  consequences of our assumption $2\alpha +\beta^2<2$.

\begin{proof}[Proof of Proposition \ref{prop:cumulbis}]
 The proof goes by induction and follows the steps of the proof of Proposition \ref{prop:cumul}.
 The key observation is that  there exists a constant
 $C_1$
 (depending on $\alpha$ and ${\bf z}$)  such that 
 for every $t\ge u$ and $x\in I$
 \begin{equation}\label{integre}
  \int_{I} (1+e^u|x-y|)^2 e^{\alpha \Upsilon(y)} \mu (\dd y)
  \le C_1 e^{-(\alpha-1)u}.
 \end{equation}
We prove \eqref{propz2} by induction.
Note that \eqref{integre} and Assumption \eqref{ass:WN} allows us to check that the statement is valid when $i=2$.
Then the  proof of the induction statement follows the steps of Proposition \ref{prop:cumul} item (2) until \eqref{bigthree} where we have to figure out how to replace the three inequalities.
We assume  that $p\ne i/2$ and $q\ge j/2$ for simplicity (the adaptation needed for the case $2p=i$ and/or $2q=j$ are exactly identical to that in the proof of Proposition \ref{prop:cumul}, and the same observation is valid for the case $j=1$).
Instead of \eqref{defH0} we need to prove 
\begin{multline}\label{defH00}
 \int_{I^{i+j-1}} \!\!\!\!\!\!\!\! e^{\alpha \Upsilon(x_k)}|Q_s(x_k,x_l)|\left[\int^\infty_s  F(\mathbf{x}_1,u)e^{\alpha \Upsilon(\bx^{(k)}_1)+\frac{\gb^2}{2}(i-1)(u-s)}  \dd u\right]\\ \times 
 \left[\int^\infty_s G(\mathbf{x}_2,v)  e^{ \alpha \Upsilon(\bx^{(k)}_2)+ \frac{\gb^2}{2}(i-1)(v-s)}\dd v \right] \mu(\dd \bx^{(k)}) \le C e^{(i+j-1)(\alpha-1)s }
\end{multline}
where $\bx^{(k)}_1$ stands for the vector $\bx_1$ with $k$-th component removed (in particular it does not change if $\bx_1$ has no $k$-th component) and similarly for $\bx^{(k)}_2$.
Now considering  $u,v>s$, we observe (using the induction hypothesis for the first and last line and \eqref{integre} for the middle one)

\begin{equation}
\begin{split}\label{bigthree2}
\int_{I^{i-1}} F(\mathbf{x}_1,u)e^{\alpha \Upsilon(\bx^{(k)}_1)} \mu(\dd \bx^{(k)}_1)&\le C_2 e^{(i-1)(\alpha-1)u},\\
   \int_{I} |Q_s(x_k,x_l)| e^{\alpha \Upsilon(x_k)} \mu(\dd x_k) &\le C_3 e^{(\alpha-1)s},\\
  \int_{I^{j-1}} G(\mathbf{x}_2,v)e ^{ \alpha \Upsilon(\bx^{(m)}_2)} \mu(\dd\bx^{(m)}_2) &\le C_4  e^{(j-1)(\alpha-1)v}.
  \end{split}
 \end{equation}
 To prove \eqref{defH00}, we just need to combine these three inequalities in the l.h.s.\ and integrate over $u,v>s$ (using that $\alpha+\frac{\gb^2}{2}<1$).
 
 \medskip
 
 \noindent Now to prove \eqref{alprev}, following the proof structure of Section \ref{sec:partition}, it is sufficient to prove that
 \begin{equation}\label{broom2}
\lim_{s\to \infty}\sup_{t\ge s}  \| \hat Q^{(n,\alpha)}_{t,t}- Q^{(n,\alpha)}_{s,t}\|_{\infty}=0.
\end{equation}
 where
 \begin{equation}\label{qnalphabis2}
 \hat Q^{(n,\alpha)}_{s,t}:= \frac{1}{2}\sum_{k=n+1}^{2n} \alpha^{k} 
   \sum_{l=k-n}^{n}\langle \hat M^{(l,t)},\hat M^{(k-l,t)} \rangle_s .
 \end{equation}
 and $\hat M^{(i,t)}$ are the martingales constructed from $\hat M$ using the iteration \ref{dafai}. These type of martingale bracket can be written in the form
 
 \begin{multline}\label{formzbis2}
\langle\hat  M^{(l,t)}, \hat M^{(k-l,t)} \rangle_s\\=\sum_{p=0}^{\lfloor k/2 \rfloor} \int^s_0 \int_{I^k} F^{(k,p)}(\bx,u,t)e^{\frac{\gb^2}{2}\sum_{l=1}^kK_u(x_l,x_l)}
 \cos\left[ \left(\sum_{j=1}^{k} \gb \lambda^{(p)}_j  X_u(x_j)+ {\bf i} \psi_t(x_j)  \right) \right]
 \mu(\dd \bx) \dd u,
\end{multline}
where $F^{(k,p)}$ verifies the properties of Propositions \ref{prop:cumul}-\ref{prop:cumulbis}.
In order to prove the convergence statement, we mostly need to prove domination. Hence we only need to observe that the cosinus term is bounded in absolute value by $C e^{\beta^2 \|\eta\|_{\infty} \Upsilon(\bx)}$, and we conclude by showing that  

\begin{equation}
 \int^\infty_0 \int_{I^k} \sup_{t\ge u} |F^{(k,p)}(\bx,u,t)|e^{\frac{\gb^2}{2}\sum_{l=1}^kK_u(x_l,x_l)}  e^{\beta^2 \|\eta\|_{\infty} \Upsilon(\bx)}
  \dd u <\infty
\end{equation}
Setting $\alpha=\|\eta_{\infty}\|\beta^2$, as a consequence of Proposition \ref{prop:cumulbis}, the above integral is smaller than
\begin{equation}
 C \int^\infty_0 e^{\frac{\gb^2}{2} ku}  e^{(k-1)(\alpha-1)u} 
  \dd u,   
\end{equation}
and the above integral is finite as soon as $(k-1)(\alpha-1)u+\frac{\gb^2}{2} ku<0$ which is valid for $n$ sufficiently large (recall that we consider only $k\ge n+1$)
as soon as $\beta^2( 1 +2 \|\eta\|_{\infty})<2$.
\end{proof}
  \appendix

\section{Explicit expression for the first two martingale brackets}
We compute here the first two brackets explicitly; we make no assumptions on the dimension $d$ and we suppose that $\beta^2>d$. We argue that up to the second threshold $\beta_2^2= \frac{3d}{2}$, the bracket $\langle M^{(1,t)}, M^{(2,t)}  \rangle_t$ is not bounded hence rendering the normalization in this case already quite involved. 
\subsubsection{Computation of $A^{(2)}$}

Recall that we have
\begin{equation}
M_t= -\gb \int_{I}  \Big (  \int_0^t   \sin (\gb X_s(x)) e^{\frac{\gb^2  }{2}K_s(x,x)}\dd X_s(x) \Big ) \mu(\dd x) - \mu (I).
\end{equation}
By transforming the product of  sines as follows 
$$2 \sin(\gb X_1)\sin(\gb X_2)= \cos \left(\gb (X_1-X_2)\right)- \cos \left(\gb( X_1+X_2)\right),$$
The definition in \eqref{defa1} and Equation \eqref{correlsnh} yield
\begin{multline}\label{computer}
 A^{(2)}_t= \frac{\gb^2}{4} \int^t_0  \int_{I^2} Q_u(x_1,x_2) e^{\frac{\gb^2}{2}(K_u(x_1,x_1)+K_u(x_2,x_2))} \\
 \times \left[ \cos \left(\gb (X_u(x_1)-X_u(x_2))\right)- \cos \left(\gb( X_u(x_1)+X_u(x_2))\right) \right]\mu(\dd x_1)\mu( \dd x_2)\dd u.
\end{multline}
\subsubsection{Computation of $M^{(2)}$}
Now, we proceed with the computation of $M^{(2)}$. Recall that for $u >s$ we have
\begin{equation*}
\E[   \cos(\beta (X_u(x)-X_u(y))) e^{\frac{\beta^2}{2} (K_u(x,x) + K_u(y,y))}  | \mathcal{F}_s  ]= \cos(\beta (X_s(x)-X_s(y))) e^{\beta^2 \int_s^u Q_v(x,y) dv } e^{\frac{\beta^2}{2} (K_s(x,x) + K_s(y,y))}
\end{equation*}
and
\begin{equation*}
\E[   \cos(\beta (X_u(x)+X_u(y))) e^{\frac{\beta^2}{2} (K_u(x,x) + K_u(y,y))}  | \mathcal{F}_s  ]= \cos(\beta (X_s(x)+X_s(y))) e^{-\beta^2 \int_s^u Q_v(x,y) dv } e^{\frac{\beta^2}{2} (K_s(x,x) + K_s(y,y))}
\end{equation*}
and therefore we get
\begin{align*}
\frac{1}{2}\E[  \langle M \rangle_t | \mathcal{F}_s  ] & = \frac{\beta^2}{4} \int_{\mathcal{O}^2}  \cos(\beta (X_s(x)-X_s(y))) e^{\frac{\beta^2}{2} (K_s(x,x) + K_s(y,y))}  \left  ( \int_s^t Q_u(x,y) e^{\beta^2 \int_s^u Q_v(x,y) dv }  du \right )   dx dy   \\
 & -\frac{\beta^2}{4}  \int_{\mathcal{O}^2}  \cos(\beta (X_s(x)+X_s(y))) e^{\frac{\beta^2}{2} (K_s(x,x) + K_s(y,y))}  \left  ( \int_s^t Q_u(x,y) e^{-\beta^2 \int_s^u Q_v(x,y) dv }  du \right )   dx dy  \\
 & +L_s \\
& = \frac{1}{4} \int_{\mathcal{O}^2}  \cos(\beta (X_s(x)-X_s(y))) e^{\frac{\beta^2}{2} (K_s(x,x) + K_s(y,y))}   ( e^{\beta^2 \int_s^t Q_v(x,y) dv } -1)   dx dy   \\
 & +\frac{1}{4}  \int_{\mathcal{O}^2}  \cos(\beta (X_s(x)+X_s(y))) e^{\frac{\beta^2}{2} (K_s(x,x) + K_s(y,y))}    (  e^{-\beta^2 \int_s^t Q_v(x,y) dv }-1  )   dx dy  \\
 & +L_s \\
\end{align*}
where
\begin{align*}
L_s= & = \frac{\beta^2}{4} \int_0^s \int_{\mathcal{O}^2}  Q_u(x,y) \cos(\beta (X_u(x)-X_u(y)))   e^{\frac{\beta^2}{2} (K_u(x,x) + K_u(y,y))} du  dx dy   \\
 & -\frac{\beta^2}{4} \int_0^s \int_{\mathcal{O}^2}  Q_u(x,y) \cos(\beta (X_u(x)+X_u(y)))   e^{\frac{\beta^2}{2} (K_u(x,x) + K_u(y,y))} du  dx dy  \\
\end{align*}
We have 
\begin{equation*}
M^{(2,t)}_s= \frac{1}{2}\E[  \langle M \rangle_t | \mathcal{F}_s  ]-\frac{1}{2}\E[  \langle M \rangle_t ] 
\end{equation*}
Using  Itô's formula, we have
\begin{align*}
& d \cos(\beta (X_s(x)-X_s(y))) e^{\frac{\beta^2}{2} (K_s(x,x) + K_s(y,y))}  \left  ( \int_s^t e^{\beta^2 \int_s^u Q_v(x,y) dv }  du \right )     \\
& = -\beta \sin(\beta (X_s(x)-X_s(y))) e^{\frac{\beta^2}{2} (K_s(x,x) + K_s(y,y))}  \left  ( \int_s^t e^{\beta^2 \int_s^u Q_v(x,y) dv }  du \right ) (dX_s(x)-dX_s(y))  + \cdots
\end{align*}
where $\cdots$ is a finite variation term. Hence we get that
\begin{align*}
& d \langle  M^{(1,t)}, M^{(2,t)}  \rangle_s   \\
& = \frac{\beta^4}{4} \int_{\mathcal{O}^3} \sin(\beta (X_s(x_1)))  \sin(\beta (X_s(x_2)-X_s(x_3)))  e^{\frac{\beta^2}{2} (K_s(x_1,x_1)+K_s(x_2,x_2) + K_s(x_3,x_3))}  \\
& \times \left  ( \int_s^t Q_u(x_2,x_3) e^{\beta^2 \int_s^u Q_v(x_2,x_3) dv } du \right ) (Q_s(x_1,x_2)-Q_s(x_1,x_3))   dx_1 dx_2 dx_3  \\
& + \frac{\beta^4}{4} \int_{\mathcal{O}^3} \sin(\beta (X_s(x_1)))  \sin(\beta (X_s(x_2)+X_s(x_3)))  e^{\frac{\beta^2}{2} (K_s(x_1,x_1)+K_s(x_2,x_2) + K_s(x_3,x_3)))} \\
& \times \left  ( \int_s^t Q_u(x_2,x_3) e^{-\beta^2 \int_s^u Q_v(x_2,x_3) dv } du \right )  (Q_s(x_1,x_2)+Q_s(x_1,x_3))  dx_1 dx_2dx_3  \\
& = \frac{\beta^2}{4} \int_{\mathcal{O}^3} \sin(\beta (X_s(x_1)))  \sin(\beta (X_s(x_2)-X_s(x_3)))  e^{\frac{\beta^2}{2} (K_s(x_1,x_1)+K_s(x_2,x_2) + K_s(x_3,x_3))}  \\
& \times ( e^{\beta^2 \int_s^t Q_v(x_2,x_3) dv }-1 ) (Q_s(x_1,x_2)-Q_s(x_1,x_3))   dx_1 dx_2 dx_3  \\
& - \frac{\beta^2}{4} \int_{\mathcal{O}^3} \sin(\beta (X_s(x_1)))  \sin(\beta (X_s(x_2)+X_s(x_3)))  e^{\frac{\beta^2}{2} (K_s(x_1,x_1)+K_s(x_2,x_2) + K_s(x_3,x_3)))} \\
& \times   (  e^{-\beta^2 \int_s^t Q_v(x_2,x_3) dv }  -1 )  (Q_s(x_1,x_2)+Q_s(x_1,x_3))  dx_1 dx_2dx_3  \\
\end{align*}

If one tries to bound the second term in the above sum, one gets by using $|\sin| \leq 1$
\begin{align*}
&  \frac{\beta^2}{4}   |  \int_{\mathcal{O}^3} \sin(\beta (X_s(x_1)))  \sin(\beta (X_s(x_2)+X_s(x_3)))  e^{\frac{\beta^2}{2} (K_s(x_1,x_1)+K_s(x_2,x_2) + K_s(x_3,x_3)))} \\
& \times   (  e^{-\beta^2 \int_s^t Q_v(x_2,x_3) dv }  -1 )  (Q_s(x_1,x_2)+Q_s(x_1,x_3))  dx_1 dx_2dx_3   |  \\
& \leq  \frac{\beta^2}{4}   \int_{\mathcal{O}^3}   e^{\frac{\beta^2}{2} (K_s(x_1,x_1)+K_s(x_2,x_2) + K_s(x_3,x_3)))} \\
& \times  (Q_s(x_1,x_2)+Q_s(x_1,x_3))  dx_1 dx_2dx_3   \\
& \leq  C e^{\frac{3 \beta^2}{2} s}   \int_{\mathcal{O}^3}   (Q_s(x_1,x_2)+Q_s(x_1,x_3))  dx_1 dx_2dx_3    \\
& \leq C e^{(\frac{3 \beta^2}{2}-d) s} 
\end{align*}
Hence, by integration over $s$ the term $\langle  M^{(1,t)}, M^{(2,t)}  \rangle_t  $ can not be bounded since $\beta^2>d$. This shows that as soon as $\beta^2>d$, one must go to order three in the renormalisation scheme; since 
\begin{align*}
 \langle \alpha M^{(1,t)} + \alpha^2 M^{(2,t)}+\alpha^3  M^{(3,t)} \rangle_t  &=  \alpha^2 \langle M^{(1,t)} \rangle_t +2 \alpha^3 \langle M^{(1,t)}, M^{(2,t)} \rangle_t+ 2 \alpha^4  \langle M^{(1,t)}, M^{(3,t)}\rangle _t  \\
 & +\alpha^4  \langle M^{(2,t)} \rangle_t+ 2\alpha^5  \langle M^{(2,t)}, M^{(3,t)} \rangle_t+ \alpha^6  \langle M^{(3,t)} \rangle_t
\end{align*}
 the cumulant decomposition yields
\begin{align*}
\E[e^{\alpha M_t}]= e^{\alpha |\mathcal{O}|  + \frac{\alpha^2}{2}\E[  \langle M \rangle_t ] +\alpha^3 \E[ \langle M^{(1)}, M^{(2)} \rangle_t   ] } & \E[ e^{ \alpha M^{(1,t)}_t+\alpha^2 M^{(2,t)}_t +\alpha^3  M^{(3,t)}_t - \frac{1}{2}  \langle \alpha M^{(1,t)} + \alpha^2 M^{(2,t)}+ \alpha^3 M^{(3,t)} \rangle_t   } \\
& \times  e^{   \frac{\alpha^4}{2} \langle M^{(2,t)} \rangle_t  +\alpha^4 \langle M^{(1,t)},M^{(3,t)}  \rangle_t  + \alpha^5  \langle M^{(2,t)}, M^{(3,t)} \rangle_t+ \frac{\alpha^6}{2}  \langle M^{(3,t)} \rangle_t  }   ]   \\
\end{align*}
and one must bound the terms $\langle M^{(2,t)} \rangle_t, \langle M^{(1,t)},M^{(3,t)}  \rangle_t , \langle M^{(2,t)}, M^{(3,t)}\rangle_t,  \langle M^{(3,t)} \rangle_t $. Explicit computations of these cases are already rather involved at this point and justify our less explicit approach.

\end{document}